\documentclass[12pt,oneside,english]{amsart}
\usepackage{lmodern}
\usepackage[T1]{fontenc}
\usepackage[latin9]{inputenc}
\usepackage{geometry}
\usepackage{babel}
\usepackage{amstext}
\usepackage{amsthm}
\usepackage{amssymb}
\usepackage[unicode=true,pdfusetitle,
 bookmarks=true,bookmarksnumbered=false,bookmarksopen=false,
 breaklinks=false,pdfborder={0 0 1},backref=false,colorlinks=false]
 {hyperref}

\makeatletter
\numberwithin{equation}{section}
\numberwithin{figure}{section}
\theoremstyle{plain}
\newtheorem{thm}{\protect\theoremname}[section]
\theoremstyle{definition}
\newtheorem{defn}[thm]{\protect\definitionname}
\theoremstyle{definition}
\newtheorem{example}[thm]{\protect\examplename}
\theoremstyle{plain}
\newtheorem{prop}[thm]{\protect\propositionname}
\ifx\proof\undefined
\newenvironment{proof}[1][\protect\proofname]{\par
\normalfont\topsep6\p@\@plus6\p@\relax
\trivlist
\itemindent\parindent
\item[\hskip\labelsep\scshape #1]\ignorespaces
}{%
\endtrivlist\@endpefalse
}
\providecommand{\proofname}{Proof}
\fi
\theoremstyle{plain}
\newtheorem{lem}[thm]{\protect\lemmaname}
\theoremstyle{plain}
\newtheorem{cor}[thm]{\protect\corollaryname}

\date{}

\title[Weak separation properties for closed subgroups]{Weak separation properties for closed subgroups of locally compact groups}

\author{Zsolt Tanko}
\address{
Department of Pure Mathematics\\
University of Waterloo\\
Waterloo, Ontario, N2L 3G1, Canada}
\email{ztanko@uwaterloo.ca}


\AtBeginDocument{
  
}

\makeatother

\providecommand{\corollaryname}{Corollary}
\providecommand{\definitionname}{Definition}
\providecommand{\examplename}{Example}
\providecommand{\lemmaname}{Lemma}
\providecommand{\propositionname}{Proposition}
\providecommand{\theoremname}{Theorem}

\begin{document}
\begin{abstract}
Three separation properties for a closed subgroup $H$ of a locally
compact group $G$ are studied: (1) the existence of a bounded approximate
indicator for $H$, (2) the existence of a completely bounded invariant
projection $VN\left(G\right)\rightarrow VN_{H}\left(G\right)$, and
(3) the approximability of the characteristic function $\chi_{H}$
by functions in $M_{cb}A\left(G\right)$ with respect to the weak$^{*}$
topology of $M_{cb}A\left(G_{d}\right)$. We show that the $H$-separation
property of Kaniuth and Lau is characterized by the existence of certain
bounded approximate indicators for $H$ and that a discretized analogue
of the $H$-separation property is equivalent to (3). Moreover, we
give a related characterization of amenability of $H$ in terms of
any group $G$ containing $H$ as a closed subgroup. The weak amenability
of $G$ or that $G_{d}$ satisfies the approximation property, in
combination with the existence of a natural projection (in the sense
of Lau and Ülger), are shown to suffice to conclude (3). Several consequences
of (2) involving the cb-multiplier completion of $A\left(G\right)$
are given. Finally, a convolution technique for averaging over the
closed subgroup $H$ is developed and used to weaken a condition for
the existence of a bounded approximate indicator for $H$.
\end{abstract}

\maketitle

\section{Introduction}

Our objective is to study connections between various forms of amenability
for a locally compact group $G$ and certain separation properties
for closed subgroups, and moreover to establish relationships between
these separation properties. Following the influential work of Ruan
\cite{Ruan}, much work on the homology of the Fourier algebra $A\left(G\right)$
as a completely contractive Banach algebra has affirmed this as the
appropriate category in which to consider $A\left(G\right)$ and the
related algebras of abstract harmonic analysis (e.g. \cite{ARS,CT,FKLS,FRS}).
Motivated by the success of this perspective, we focus on completely
bounded projections, operator amenability, and the completely bounded
multiplier algebra of $A\left(G\right)$. We consider the following
separation properties for a closed subgroup $H$ of $G$:
\begin{enumerate}
\item \label{enu:Sep_prop_BAI}The existence of a bounded approximate indicator
for $H$.
\item \label{enu:Sep_prop_InvCompl}The existence of a completely bounded
$A\left(G\right)$-bimodule projection of $VN\left(G\right)$ onto
$VN_{H}\left(G\right)$.
\item \label{enu:Sep_prop_CharFuncApprox}When the characteristic function
$\chi_{H}$ may be approximated by functions in $B\left(G\right)$
or $M_{cb}A\left(G\right)$ in the weak$^{*}$ topology on the corresponding
algebra of the discretized group.
\end{enumerate}
Condition \eqref{enu:Sep_prop_CharFuncApprox} is also considered
for subsets of $G$ that are not necessarily closed subgroups.

Bounded approximate indicators for closed subgroups were introduced
in \cite{ARS} as a means of obtaining invariant projections. They
have subsequently been shown to have an intimate connection with homological
properties of $A\left(G\right)$ and its completion in the cb-multipliers
$A_{cb}\left(G\right)$ \cite{CT}. A result of Granirer and Leinert
\cite[Theorem B2]{GL} yields bounded approximate indicators in $B\left(G\right)$
from weaker conditions than given in Definition \ref{def:Subgroup_approx_properties}.
This useful tool is unavailable for nets in $M_{cb}A\left(G\right)$
and Section \ref{sec:Averaging_over_subgroups} develops a convolution
technique relative to the closed subgroup $H$ that recovers the weakened
condition in the cb-multiplier setting.

The existence of invariant projections has been studied by several
authors in connection with other separation properties and the existence
of approximate identities for ideals in $A\left(G\right)$ \cite{ARS,CT,Derighetti,Forrest1,FKLS,KanLau}.
In particular, in \cite{KanLau} it is shown that if $G$ has the
$H$-separation property, then an invariant projection $VN\left(G\right)\rightarrow VN_{H}\left(G\right)$
exists. We show in Section \ref{sec:The_discr_sep_prop} that, in
fact, the $H$-separation property is equivalent to the existence
of a bounded approximate indicator for $H$ consisting of positive
definite functions that are identically one on $H$. An analogue of
the $H$-separation property is moreover shown to characterize the
$B\left(G\right)$-approximability of $\chi_{H}$. Condition \eqref{enu:Sep_prop_CharFuncApprox}
was first studied in \cite{ARS}, where it was claimed that a bounded
approximate indicator for $H$ exists whenever $\chi_{H}$ is $B\left(G\right)$-approximable.
This argument was later found to contain a gap \cite{ARST}. We give
examples in Section \ref{sec:Approximability_of_char_funs} showing
that the cb-multiplier analogue is false.

Conditions (1) to (3) are related to amenability properties of $G$
and to homological properties of $A\left(G\right)$, and it is this
connection that is the main focus of the present article. We show
in Section \ref{sec:The_discr_sep_prop} that $H$ is already amenable
when $\chi_{H}$ is $A\left(G\right)$-approximable for any locally
compact group $G$ containing $H$ as a closed subgroup. In the case
that $G$ is amenable, the algebra $A\left(G\right)$ has a bounded
approximate identity and \cite[Proposition 6.4]{Forrest1} then asserts
that an invariant projection $VN\left(G\right)\rightarrow VN_{H}\left(G\right)$
exists exactly when the ideal $I_{A\left(G\right)}\left(H\right)$
in $A\left(G\right)$ has a bounded approximate identity. By \cite{FKLS},
the latter occurs for every closed subgroup of $G$ and it follows
that an approximate indicator exists for every closed subgroup, since
$\left(1_{G}-e_{\alpha}\right)_{\alpha}$ is an approximate indicator
for $H$ when $\left(e_{\alpha}\right)_{\alpha}$ is a bounded approximate
identity for $I_{A\left(G\right)}\left(H\right)$. Thus all closed
subgroups of an amenable locally compact group are separated in the
strongest sense that we consider. For generic locally compact groups
the situation is more complicated, although some strong connections
are known to hold in general. For example, using the identity $A\left(G\right)\widehat{\otimes}A\left(G\right)=A\left(G\times G\right)$,
it is routine to show that an approximate indicator for the diagonal
$G_{\Delta}$ in $A\left(G\times G\right)$ (and bounded there) is
exactly a bounded approximate diagonal for $A\left(G\right)$, the
existence of which characterizes amenability of $G$ \cite{Ruan}.
Moreover, the existence of an invariant projection $VN\left(G\times G\right)\rightarrow VN_{G_{\Delta}}\left(G\times G\right)$
characterizes the operator biflatness of $A\left(G\right)$, a weaker
homological condition than operator amenability. Contractive operator
biflatness, which asks that this invariant projection to be a complete
contraction, was recently shown to be equivalent to the existence
of a contractive approximate indicator for $G_{\Delta}$ in $B\left(G\times G\right)$
\cite{CT}.

A bounded approximate indicator for $H$ always yields the approximability
of $\chi_{H}$ in the corresponding algebra, however it is unclear
when the latter follows from the existence of an invariant projection
$VN\left(G\right)\rightarrow VN_{H}\left(G\right)$ alone. In Section
\ref{sec:Approximability_of_char_funs} we show that if $H$ satisfies
certain weak forms of amenability, then the existence of a bounded
map $VN\left(G\right)\rightarrow VN_{H}\left(G\right)$ satisfying
$\lambda\left(s\right)\mapsto\chi_{H}\left(s\right)\lambda\left(s\right)$
\textemdash{} a weaker condition than the existence of an invariant
projection onto $VN_{H}\left(G\right)$ \textemdash{} implies the
$M_{cb}A\left(G\right)$-approximability of $\chi_{H}$. We give an
example in which the former condition fails while the latter holds.
Establishing relations amongst the conditions (1) to (3) in the presence
of amenability type conditions on $H$ or $G$ is the second goal
of this article.

\section{Preliminaries}

For a locally compact group $G$, the following algebras were defined
by Eymard in \cite{Eymard}, who established the basic properties
we outline below. The space of coefficient functions of strongly continuous
unitary representations of $G$,
\begin{equation}
B\left(G\right)=\left\{ \left\langle \pi\left(\cdot\right)\xi|\eta\right\rangle :\pi:G\rightarrow B\left(\mathcal{H}\right)\mbox{ is a representation, }\xi,\eta\in\mathcal{H}\right\} ,\label{eq:Fourier_Stieltjes_algebra}
\end{equation}
forms a commutative completely contractive Banach algebra, the \emph{Fourier\textendash \hspace{1mm}Stieltjes
algebra} of $G$, under pointwise multiplication and norm
\[
\left\Vert u\right\Vert _{B\left(G\right)}=\inf\left\{ \left\Vert \xi\right\Vert \left\Vert \eta\right\Vert :u=\left\langle \pi\left(\cdot\right)\xi|\eta\right\rangle \mbox{ with }\pi,\xi,\eta\mbox{ as in }\eqref{eq:Fourier_Stieltjes_algebra}\right\} .
\]
The operator space structure on $B\left(G\right)$ arises from its
identification with the dual space of the universal enveloping $C^{*}$-algebra
of $L^{1}\left(G\right)$ \textemdash{} the \emph{group $C^{*}$-algebra}
$C^{*}\left(G\right)$ of $G$ \textemdash{} via the duality
\begin{eqnarray*}
\left\langle u,\pi\left(f\right)\right\rangle _{B\left(G\right),C^{*}\left(G\right)}=\int_{G}fu &  & \left(u\in B\left(G\right),f\in L^{1}\left(G\right)\right).
\end{eqnarray*}
We refer to \cite{ER} for the theory of operator spaces and completely
contractive Banach algebras. The \emph{positive definite functions}
in $B\left(G\right)$ are those that correspond to positive functionals
on $C^{*}\left(G\right)$ and are denoted $P\left(G\right)$. For
$u\in P\left(G\right)$ we have $\left\Vert u\right\Vert _{B\left(G\right)}=\left\Vert u\right\Vert _{L^{\infty}\left(G\right)}=u\left(e\right)$.
The adjoint on $B\left(G\right)$ is given by
\begin{eqnarray*}
u^{*}\left(s\right)=\overline{u\left(s^{-1}\right)} &  & \left(u\in B\left(G\right),s\in G\right)
\end{eqnarray*}
and the self-adjoint functions in $B\left(G\right)$ correspond to
the self-adjoint bounded functionals on $C^{*}\left(G\right)$. Consequently,
given $u\in B\left(G\right)$ self-adjoint, there exist $u^{\pm}\in P\left(G\right)$
such that $u=u^{+}-u^{-}$ and $\left\Vert u\right\Vert _{B\left(G\right)}=\left\Vert u^{+}\right\Vert _{B\left(G\right)}+\left\Vert u^{-}\right\Vert _{B\left(G\right)}$.

The \emph{Fourier algebra} of $G$ is the closed ideal $A\left(G\right)$
of $B\left(G\right)$ given by the coefficients of the left regular
representation $\lambda:G\rightarrow B\left(L^{2}\left(G\right)\right)$,
which is defined by
\begin{eqnarray*}
\lambda\left(s\right)\xi\left(t\right)=\xi\left(s^{-1}t\right) &  & \left(s,t\in G,\xi\in L^{2}\left(G\right)\right).
\end{eqnarray*}
When necessary, we denote this representation of $G$ by $\lambda_{G}$.
The Fourier algebra coincides with the closure of the compactly supported
functions in $B\left(G\right)$, is closed under the adjoint, and
is regular in the sense that for any $K\subset G$ compact and $U\supset K$
open, there is a function $v\in A\left(G\right)$ with $v\left(K\right)=1$
and $v\left(G\setminus U\right)=0$. The \emph{group von Neumann algebra}
of $G$ is the weak operator topology closure $VN\left(G\right)$
of $\text{span}\lambda\left(G\right)$ in $B\left(L^{2}\left(G\right)\right)$
and is identified with the dual of the Fourier algebra via
\begin{eqnarray*}
\left\langle \lambda\left(s\right),u\right\rangle _{VN\left(G\right),A\left(G\right)}=u\left(s\right) &  & \left(s\in G,u\in A\left(G\right)\right).
\end{eqnarray*}

Given an algebra $\mathcal{A}$ of functions on $G$ and a subset
$E$ of $G$, we denote by $I_{\mathcal{A}}\left(E\right)$ the ideal
of functions in $\mathcal{A}$ vanishing on $E$. For a closed subgroup
$H$ of $G$, the annihilator $I_{A\left(G\right)}\left(H\right)^{\perp}$
coincides with the von Neumann algebra $VN_{H}\left(G\right)$ generated
by $\lambda_{G}\left(H\right)$ \cite[Theorem 6]{TT2}, which is identified
with $VN\left(H\right)$ via the normal $\ast$-isomorphism defined
by $\lambda_{H}\left(h\right)\mapsto\lambda_{G}\left(h\right)$ for
$h\in H$. The preadjoint of this normal $\ast$-isomorphism is the
restriction map $r_{H}:A\left(G\right)\rightarrow A\left(H\right)$,
which is thus a complete quotient. The closed subgroups of $G$ are
sets of \emph{spectral synthesis} for $A\left(G\right)$ \cite{Herz},
meaning that the ideal $I_{A\left(G\right)}\left(H\right)$ of $A\left(G\right)$
coincides with the closure of the functions in $A\left(G\right)$
that have compact support disjoint from $H$.

\emph{The completely bounded (cb-) multiplier algebra} $M_{cb}A\left(G\right)$
of $A\left(G\right)$ is the algebra of functions $m$ on $G$ for
which $mA\left(G\right)\subset A\left(G\right)$ and the map $M_{m}:VN\left(G\right)\rightarrow VN\left(G\right)$
is completely bounded, where $M_{m}$ is the adjoint of the multiplication
map $u\mapsto mu$ on $A\left(G\right)$. Such functions are continuous
and bounded, and form a completely contractive Banach algebra under
pointwise multiplication and norm $\left\Vert m\right\Vert _{M_{cb}A\left(G\right)}=\left\Vert M_{m}\right\Vert _{cb}$.
The cb-multipliers of $A\left(G\right)$ admit the following representation
theorem \cite{Jolissaint}.
\begin{thm}
\emph{\label{thm:Gilbert}(Gilbert's representation theorem)} Let
$G$ be a locally compact group. A function $m$ on $G$ is in $M_{cb}A\left(G\right)$
if and only if there exists a Hilbert space $\mathcal{H}$ and bounded
continuous maps $P,Q:G\rightarrow\mathcal{H}$ such that 
\begin{eqnarray*}
m\left(s^{-1}t\right)=\left\langle P\left(t\right)|Q\left(s\right)\right\rangle  &  & \left(s,t\in G\right).
\end{eqnarray*}
 The norm $\left\Vert m\right\Vert _{M_{cb}A\left(G\right)}$ is the
infimum of the quantities $\left\Vert P\right\Vert _{\infty}\left\Vert Q\right\Vert _{\infty}$
taken over all such maps $P$ and $Q$ and Hilbert spaces $\mathcal{H}$.
\end{thm}
It follows from Gilbert's representation theorem that $B\left(G\right)\subset M_{cb}A\left(G\right)$
and that $\left\Vert \cdot\right\Vert _{M_{cb}A\left(G\right)}\leq\left\Vert \cdot\right\Vert _{B\left(G\right)}$
on $B\left(G\right)$. The restriction $r_{H}:M_{cb}A\left(G\right)\rightarrow M_{cb}A\left(H\right)$
is a well defined complete contraction \cite[Proposition 1.12]{dCH}.
The norm closure of $A\left(G\right)$ in the potentially smaller
cb-multiplier norm is denoted $A_{cb}\left(G\right)$ and forms a
completely contractive Banach subalgebra of the cb-multipliers with
spectrum $G$ \cite[Proposition 2.2]{FRS}.

Since cb-multipliers of $A\left(G\right)$ lie in $L^{\infty}\left(G\right)$,
we may consider $L^{1}\left(G\right)$ as a subspace of the dual of
$M_{cb}A\left(G\right)$. Taking the completion of $L^{1}\left(G\right)$
with respect to the norm given, for $f\in L^{1}\left(G\right)$, by
\[
\left\Vert f\right\Vert _{Q\left(G\right)}=\sup\left\{ \left|\int_{G}fm\right|:m\in M_{cb}A\left(G\right)\mbox{ with }\left\Vert m\right\Vert _{M_{cb}A\left(G\right)}\leq1\right\} 
\]
yields a predual $Q\left(G\right)$ for $M_{cb}A\left(G\right)$ \cite[Proposition 1.10]{dCH}.
With this predual, the cb-multipliers $M_{cb}A\left(G\right)$ form
a completely contractive dual Banach algebra in the sense of Runde
\cite{Runde1}. It follows from Theorem \ref{thm:Gilbert} that $\left\Vert \cdot\right\Vert _{L^{\infty}\left(G\right)}\leq\left\Vert \cdot\right\Vert _{M_{cb}A\left(G\right)}$
and consequently $\left\Vert \cdot\right\Vert _{Q\left(G\right)}\leq\left\Vert \cdot\right\Vert _{L^{1}\left(G\right)}$
on $L^{1}\left(G\right)$. We let $\mathcal{C}_{c}\left(G\right)$
and $\mathcal{C}_{b}\left(G\right)$ denote respectively the continuous
compactly supported and continuous bounded functions of $G$. We have
\begin{eqnarray*}
A\left(G\right)\subset A_{cb}\left(G\right) & \mbox{and} & A\left(G\right)\subset B\left(G\right)\subset M_{cb}A\left(G\right)\subset\mathcal{C}_{b}\left(G\right).
\end{eqnarray*}
The second containment is strict unless $G$ is compact. An unpublished
result of Ruan asserts that $A\left(G\right)$ is closed in $M_{cb}A\left(G\right)$
exactly when $G$ is amenable, so that that the first and third containments
are strict unless $G$ is amenable. For $u\in A\left(G\right)$ and
$v\in B\left(G\right)$, 
\[
\left\Vert u\right\Vert _{\infty}\leq\left\Vert u\right\Vert _{A_{cb}\left(G\right)}=\left\Vert u\right\Vert _{M_{cb}A\left(G\right)}\leq\left\Vert u\right\Vert _{A\left(G\right)}=\left\Vert u\right\Vert _{B\left(G\right)},
\]
\[
\left\Vert v\right\Vert _{\infty}\leq\left\Vert v\right\Vert _{M_{cb}A\left(G\right)}\leq\left\Vert v\right\Vert _{B\left(G\right)},
\]
so the first, third, and fourth inclusions are in general contractive
while the second is isometric.

The locally compact group $G$ equipped with the discrete topology
is denoted $G_{d}$. The inclusions $B\left(G\right)\subset B\left(G_{d}\right)$
and $M_{cb}A\left(G\right)\subset M_{cb}A\left(G_{d}\right)$ are
complete isometries (\cite{Eymard} and \cite[Corollary 6.3]{Spronk},
respectively). For each $s\in G$ the point mass $\delta_{s}\in\ell^{1}\left(G_{d}\right)$
is contained in $C^{*}\left(G_{d}\right)$ and in $Q\left(G_{d}\right)$
as the evaluation functional at $s$, from which it follows that convergence
in the $\mbox{weak}^{*}$ topology of $B\left(G_{d}\right)$ or $M_{cb}A\left(G_{d}\right)$
implies pointwise convergence. It will be important for us that, on
bounded sets, the converse holds (see \cite{Eymard} regarding $B\left(G_{d}\right)$
and \cite[Lemma 2.6]{FRS} or the useful Appendix A of \cite{Knudby}
regarding $M_{cb}A\left(G_{d}\right)$).

The separation properties for closed subgroups that we discuss are
defined as follows.
\begin{defn}
\label{def:Subgroup_approx_properties}Let $G$ be a locally compact
group and $H$ a closed subgroup.
\begin{enumerate}
\item A\textbf{\emph{ }}\emph{bounded approximate indicator} for $H$ is
a bounded net $\left(m_{\alpha}\right)_{\alpha}$ in $M_{cb}A\left(G\right)$
satisfying
\begin{enumerate}
\item $\left\Vert ur_{H}\left(m_{\alpha}\right)-u\right\Vert _{A\left(H\right)}\rightarrow0$
for all $u\in A\left(H\right)$, and
\item $\left\Vert um_{\alpha}\right\Vert _{A\left(G\right)}\rightarrow0$
for all $u\in I_{A\left(G\right)}\left(H\right)$.
\end{enumerate}
If $\left(m_{\alpha}\right)_{\alpha}$ is in $B\left(G\right)$ and
is also bounded there, then we refer to a bounded approximate indicator
in $B\left(G\right)$.
\item A completely bounded projection $VN\left(G\right)\rightarrow VN_{H}\left(G\right)$
is \emph{invariant} if it is an $A\left(G\right)$-bimodule map.
\item Given a norm closed subalgebra $\mathcal{A}$ of $B\left(G\right)$
or $M_{cb}A\left(G\right)$ and $E\subset G$, the characteristic
function $\chi_{E}$ is called \emph{$\mathcal{A}$-approximable}
if $\chi_{E}$ is in the weak$^{*}$ closure of $\mathcal{A}$ in
$B\left(G_{d}\right)$ or $M_{cb}A\left(G_{d}\right)$, respectively.
\end{enumerate}
\end{defn}
The \emph{operator amenability} of the Fourier algebra asserts the
existence of a \emph{bounded approximate diagonal} in the operator
space projective tensor product $A\left(G\right)\widehat{\otimes}A\left(G\right)$.
This is a bounded net $\left(d_{\alpha}\right)_{\alpha}$ in the tensor
product which satisfies the norm convergence
\begin{eqnarray*}
u\cdot d_{\alpha}-d_{\alpha}\cdot u\rightarrow0\text{ and }\Delta\left(d_{\alpha}\right)u\rightarrow u &  & \left(u\in A\left(G\right)\right),
\end{eqnarray*}
where $\Delta:A\left(G\right)\widehat{\otimes}A\left(G\right)\rightarrow A\left(G\right)$
is the completely bounded linearization of the multiplication and
the $A\left(G\right)$ action on the tensor product is given by $u\cdot\left(v\otimes w\right)=uv\otimes w$
and $\left(v\otimes w\right)\cdot u=v\otimes wu$, for $u,v,w\in A\left(G\right)$.
The product map $\Delta$ is a complete quotient and the \emph{operator
biflatness} of $A\left(G\right)$ asserts the existence of a completely
bounded $A\left(G\right)$-bimodule left inverse to its adjoint $\Delta^{*}:VN\left(G\right)\rightarrow VN\left(G\right)\overline{\otimes}VN\left(G\right)$,
where the $A\left(G\right)$ action on $VN\left(G\right)$ is the
dual action. The identification $A\left(G\right)\widehat{\otimes}A\left(G\right)=A\left(G\times G\right)$
(see \cite[Theorem 7.2.4]{ER}) yields $\ker\Delta=I_{A\left(G\times G\right)}\left(G_{\Delta}\right)$,
from which it follows that such a left inverse is exactly an invariant
projection $VN\left(G\times G\right)\rightarrow VN_{G_{\Delta}}\left(G\times G\right)$.
A thorough account of the homological conditions we consider is given
in \cite{Runde}.

Leptin's classical result states that amenability of a locally compact
group $G$ is characterized by the existence of a bounded approximate
identity in $A\left(G\right)$ \cite{Leptin}. The locally compact
group $G$ is called \emph{weakly amenable} when $A_{cb}\left(G\right)$
has a bounded approximate identity. This weaker notion was introduced
by de Cannière and Haagerup \cite{dCH}, who showed that the free
group on two generators is weakly amenable. The weak amenability of
$G$ is equivalent to the assertion that every cb-multiplier is the
weak$^{*}$ limit of a bounded net in $A\left(G\right)$. When $A\left(G\right)$
is merely weak$^{*}$ dense in $M_{cb}A\left(G\right)$, the group
$G$ is said to have the \emph{approximation property} (see \cite{HK}).

\section{\label{sec:Approximability_of_char_funs}Approximability of characteristic
functions}

In this section, we investigate when characteristic functions of subsets
of a locally compact group $G$ are approximable. For a closed subgroup
$H$ of $G$, the assertion that $\chi_{H}$ is approximable may be
viewed as a very weak form of subgroup separation. In \cite{ARS},
the \emph{discretized Fourier\textendash Stieltjes algebra} $B^{d}\left(G\right)$
is defined to be the $\mbox{weak}^{*}$ closure of $B\left(G\right)$
in $B\left(G_{d}\right)$. We make the analogous definition for the
cb-multipliers of $G$.
\begin{defn}
The \emph{discretized cb-multiplier algebra} $M_{cb}^{d}A\left(G\right)$
of a locally compact group $G$ is the weak$^{*}$ closure of $M_{cb}A\left(G\right)$
in $M_{cb}A\left(G_{d}\right)$. The predual $Q\left(G_{d}\right)/M_{cb}^{d}\left(G\right)_{\perp}$
of $M_{cb}^{d}A\left(G\right)$ is denoted $Q^{d}\left(G\right)$.
Let $A^{d}\left(G\right)$ and $A_{cb}^{d}\left(G\right)$ denote
the weak$^{*}$ closures of $A\left(G\right)$ in $B\left(G_{d}\right)$
and in $M_{cb}A\left(G_{d}\right)$, respectively.
\end{defn}
Given $E\subset G$, for $\chi_{E}$ to be approximable, it must already
be that $\chi_{E}\in M_{cb}A\left(G_{d}\right)$, and the subsets
of $G$ for which this occurs are not well understood when $G_{d}$
is not amenable. In the amenable case, the algebras $M_{cb}A\left(G_{d}\right)$
and $B\left(G_{d}\right)$ coincide \cite{Losert} and the Cohen-Host
idempotent theorem provides a complete description of the subsets
of $G$ with characteristic function in $B\left(G_{d}\right)$. For
discrete groups $G$, determining the approximable characteristic
functions is exactly the problem of determining the sets with characteristic
function in the cb-multipliers. When $G$ is moreover weakly amenable,
Corollary 5.4 of \cite{CT} together with Corollary 3.5 of \cite{FRS}
implies that such sets $E$ are exactly those for which the ideal
$I_{A\left(G\right)}\left(E\right)$ has a cb-multiplier bounded approximate
identity.
\begin{example}
\label{exa:BAI_implies_approximable}Let $G$ be a locally compact
group and $H$ a closed subgroup for which a bounded approximate indicator
$\left(m_{\alpha}\right)_{\alpha}$ exists. We show that $\chi_{H}$
is approximable. If $s\in H$, then we may find $u\in A\left(H\right)$
with $u\left(s\right)=1$, in which case 
\[
\left|m_{\alpha}\left(s\right)-1\right|\leq\left\Vert ur_{H}\left(m_{\alpha}\right)-u\right\Vert _{L^{\infty}\left(H\right)}\leq\left\Vert ur_{H}\left(m_{\alpha}\right)-u\right\Vert _{A\left(H\right)}\rightarrow0.
\]
If $s\in G\setminus H$, then \cite[Lemme 3.2]{Eymard} asserts that
we may find $w\in I_{A\left(G\right)}\left(H\right)$ with $w\left(s\right)=1$,
and
\[
\left|m_{\alpha}\left(s\right)\right|\leq\left\Vert wm_{\alpha}\right\Vert _{L^{\infty}\left(G\right)}\leq\left\Vert wm_{\alpha}\right\Vert _{A\left(G\right)}\rightarrow0.
\]
Since $M_{cb}A\left(G\right)$ is contained in $M_{cb}A\left(G_{d}\right)$
and weak$^{*}$ and pointwise convergence coincide on bounded subsets
of $M_{cb}A\left(G_{d}\right)$, it follows that $\left(m_{\alpha}\right)_{\alpha}$
has weak$^{*}$ limit $\chi_{H}$ in $M_{cb}A\left(G_{d}\right)$.
\end{example}
The algebra $M_{cb}^{d}A\left(G\right)$ is a $\mbox{weak}^{*}$ closed
subalgebra of $M_{cb}A\left(G_{d}\right)$ and thus has separately
$\mbox{weak}^{*}$ continuous multiplication. If we impose a rather
weak condition on the discrete group $G_{d}$, then every function
in $M_{cb}A\left(G_{d}\right)$ may be approximated in the weak$^{*}$
topology by functions in $A\left(G\right)$.\sloppy
\begin{prop}
\label{pro:Approx_prop_implies_approximable}Let $G$ be a locally
compact group. The inclusion $A_{cb}\left(G_{d}\right)\subset A_{cb}^{d}\left(G\right)$
always holds. Consequently, if $G_{d}$ has the approximation property,
then $M_{cb}A\left(G_{d}\right)=A_{cb}^{d}\left(G\right)$.
\end{prop}
\begin{proof}
By Proposition 1 of \cite{BKLS} we have $A\left(G_{d}\right)\subset A^{d}\left(G\right)$
and, because weak$^{*}$ convergence in $B\left(G_{d}\right)$ implies
weak$^{*}$ convergence in $M_{cb}A\left(G_{d}\right)$, it follows
that $A\left(G_{d}\right)\subset A_{cb}^{d}\left(G\right)$. Taking
norm closures in $M_{cb}A\left(G_{d}\right)$ yields $A_{cb}\left(G_{d}\right)\subset A_{cb}^{d}\left(G\right)$.
If $\left(e_{\alpha}\right)_{\alpha}$ is a net in $A\left(G_{d}\right)$
converging weak$^{*}$ to $1_{G}$ in $M_{cb}A\left(G_{d}\right)$
and if $m\in M_{cb}A\left(G_{d}\right)$, then $me_{\alpha}\in A\left(G_{d}\right)$
for each $\alpha$ and $me_{\alpha}\overset{w^{*}}{\rightarrow}m$
in $M_{cb}A\left(G_{d}\right)$, implying that $m\in A_{cb}^{d}\left(G\right)$.
\end{proof}
In Section 3 of \cite{ARS}, that the Fourier\textendash Stieltjes
algebra is the dual of a $C^{*}$-algebra is noted to imply that the
unique $\mbox{weak}^{*}$ continuous extension of the inclusion $B\left(G\right)\subset B^{d}\left(G\right)$
is a quotient map $B\left(G\right)^{**}\rightarrow B^{d}\left(G\right)$.
We provide a more concrete construction of the analogous canonical
map $M_{cb}A\left(G\right)^{**}\rightarrow M_{cb}^{d}A\left(G\right)$
that exploits the relation between $M_{cb}A\left(G\right)$ and $M_{cb}^{d}A\left(G\right)$.
Let $\iota_{d}:M_{cb}A\left(G\right)\rightarrow M_{cb}A\left(G_{d}\right)$
and $\kappa_{Q}:Q\left(G_{d}\right)\rightarrow Q\left(G_{d}\right)^{**}$
denote the inclusion maps. By the bipolar theorem $M_{cb}^{d}A\left(G\right)_{\perp}=M_{cb}A\left(G\right)_{\perp}$
and we have 
\[
\kappa_{Q}\left(M_{cb}A\left(G\right)_{\perp}\right)\subset M_{cb}A\left(G\right)^{\perp}=\mbox{im}\left(\iota_{d}\right)^{\perp}=\ker\left(\iota_{d}^{*}\right),
\]
which together imply that the composition
\[
Q\left(G_{d}\right)\overset{\kappa_{Q}}{\rightarrow}Q\left(G_{d}\right)^{**}\overset{\iota_{d}^{*}}{\rightarrow}Q\left(G\right)^{**}
\]
induces a map $Q^{d}\left(G\right)\rightarrow Q\left(G\right)^{**}$.
Denote the adjoint of this induced map by $\tau:M_{cb}A\left(G\right)^{**}\rightarrow M_{cb}^{d}A\left(G\right)$.
It is straightforward to verify that $\tau$ extends inclusion $M_{cb}A\left(G\right)\subset M_{cb}^{d}A\left(G\right)$,
so that $\tau\left(m\right)\left(s\right)=m\left(s\right)$ for $m\in M_{cb}A\left(G\right)$
and $s\in G$. It follows that if a net $\left(m_{\alpha}\right)_{\alpha}$
in $M_{cb}A\left(G\right)$ converges weak$^{*}$ to $\omega\in M_{cb}A\left(G\right)^{**}$,
then
\begin{eqnarray*}
\tau\left(\omega\right)\left(s\right)=\lim_{\alpha}m_{\alpha}\left(s\right) &  & \left(s\in G\right),
\end{eqnarray*}
so that the map $\tau$ extracts pointwise limits from nets in $M_{cb}A\left(G\right)$
that are weak$^{*}$ convergent in the bidual. Moreover, the range
of $\tau$ consists of exactly those functions in $M_{cb}^{d}A\left(G\right)$
that are limits of bounded nets in $M_{cb}A\left(G\right)$.
\begin{prop}
\label{pro:chi_H_approximable}Let $G$ be a locally compact group
and $E\subset G$. If there is a bounded map $\Psi:VN\left(G\right)\rightarrow VN\left(G\right)$
satisfying $\Psi\left(\lambda\left(s\right)\right)=\chi_{E}\left(s\right)\lambda\left(s\right)$
for all $s\in G$, then $\chi_{E}A\left(G\right)\subset A_{cb}^{d}\left(G\right)$.
If, moreover, $\chi_{E}\in M_{cb}A\left(G_{d}\right)$ and $1_{G}$
is $A_{cb}\left(G\right)$-approximable, then $\chi_{E}$ is $A_{cb}\left(G\right)$-approximable.
\end{prop}
\begin{proof}
Let $\kappa_{A}:A\left(G\right)\rightarrow A\left(G\right)^{**}$
and $\iota_{A}:A\left(G\right)\rightarrow M_{cb}A\left(G\right)$
be the inclusions and let $\sigma$ denote the composition
\[
A\left(G\right)\overset{\kappa_{A}}{\longrightarrow}A\left(G\right)^{**}\overset{\Psi^{*}}{\longrightarrow}A\left(G\right)^{**}\overset{\iota_{A}^{**}}{\longrightarrow}M_{cb}A\left(G\right)^{**}\overset{\tau}{\longrightarrow}M_{cb}^{d}A\left(G\right).
\]
 For $u\in A\left(G\right)$ and $s\in G$, with $\delta_{s}$ denoting
the point mass at $s$ in $\ell^{1}\left(G_{d}\right)\subset Q\left(G_{d}\right)\subset Q^{d}\left(G\right)$,
\begin{eqnarray*}
\sigma\left(u\right)\left(s\right) & = & \left\langle \sigma\left(u\right),\delta_{s}\right\rangle _{M_{cb}^{d}A\left(G\right),Q^{d}\left(G\right)}\\
 & = & \left\langle \iota_{A}^{**}\Psi^{*}\kappa_{A}\left(u\right),\iota_{d}^{*}\kappa_{Q}\left(\delta_{s}\right)\right\rangle _{M_{cb}A\left(G\right)^{**},M_{cb}A\left(G\right)^{*}}\\
 & = & \left\langle \Psi^{*}\kappa_{A}\left(u\right),\lambda\left(s\right)\right\rangle _{A\left(G\right)^{**},VN\left(G\right)}\\
 & = & \left\langle \Psi\left(\lambda\left(s\right)\right),u\right\rangle _{VN\left(G\right),A\left(G\right)}\\
 & = & \chi_{E}\left(s\right)u\left(s\right).
\end{eqnarray*}
Thus $\chi_{E}u=\sigma\left(u\right)\in M_{cb}^{d}A\left(G\right)$
for all $u\in A\left(G\right)$. Since $\tau$ extends the inclusion
of $M_{cb}A\left(G\right)$ into $M_{cb}^{d}A\left(G\right)$, it
maps $A\left(G\right)$ into $A_{cb}^{d}\left(G\right)$, which together
with the weak$^{*}$ continuity of $\tau\iota_{A}^{**}$ implies that
$\sigma$ in fact has range in $A_{cb}^{d}\left(G\right)$. If $\left(e_{\alpha}\right)_{\alpha}$
is a net in $A\left(G\right)$ converging weak$^{*}$ to $1_{G}$
in $M_{cb}A\left(G_{d}\right)$, then that $\chi_{E}\in M_{cb}A\left(G_{d}\right)$
implies $\chi_{E}e_{\alpha}\overset{w^{*}}{\rightarrow}\chi_{E}$,
by weak$^{*}$ continuity of multiplication in $M_{cb}A\left(G_{d}\right)$.
Since $\chi_{E}e_{\alpha}\in A_{cb}^{d}\left(G\right)$, we conclude
that $\chi_{E}$ is $A_{cb}\left(G\right)$-approximable.
\end{proof}
For a subgroup $H$ of a locally compact group $G$, it is straight
forward to verify that $\chi_{H}$ is a positive definite function
on $G_{d}$, so that $\chi_{H}\in B\left(G_{d}\right)\subset M_{cb}A\left(G_{d}\right)$
and the second part of Proposition \ref{pro:chi_H_approximable} is
applicable to characteristic functions of subgroups. It is shown in
Lemma \ref{lem:Weaker_cond_for_chiH_approximable} below that, when
$\chi_{H}A\left(G\right)\subset A_{cb}^{d}\left(G\right)$, we need
only require $1_{H}$ to be $A_{cb}\left(H\right)$-approximable to
deduce that $\chi_{H}$ is $A_{cb}\left(G\right)$-approximable.

In \cite{L=0000DC}, Lau and Ülger define a projection $\Psi$ on
$VN\left(G\right)$ to be \emph{natural} if $\Psi\left(\lambda\left(s\right)\right)=\chi_{E}\left(s\right)\lambda\left(s\right)$
for some subset $E$ of $G$. We may interpret Proposition \ref{pro:chi_H_approximable}
as imposing restrictions on which subsets of $G$ can arise from a
natural projection.
\begin{example}
\label{exa:Inv_proj_are_natural}Let $G$ be a locally compact group
and $H$ a closed subgroup. We show that an invariant projection $\Psi:VN\left(G\right)\rightarrow VN_{H}\left(G\right)$
is natural. It is clear that $\Psi\left(\lambda\left(s\right)\right)=\lambda\left(s\right)$
for $s\in H$. Let $s\in G\setminus H$ and let $T_{\alpha}\in\mbox{span}\lambda\left(H\right)$
converge weak$^{*}$ to $\Psi\left(\lambda\left(s\right)\right)$
in $VN\left(G\right)$. If $u\in A\left(G\right)$ with $u\left(s\right)=1$
and $\left.u\right|_{H}=0$, then 
\begin{eqnarray*}
\left\langle u\cdot\lambda\left(s\right),v\right\rangle =\left\langle \lambda\left(s\right),vu\right\rangle =v\left(s\right)u\left(s\right)=v\left(s\right)=\left\langle \lambda\left(s\right),v\right\rangle  &  & \left(v\in A\left(G\right)\right),
\end{eqnarray*}
so $u\cdot\lambda\left(s\right)=\lambda\left(s\right)$. If $S=\sum_{j}\alpha_{j}\lambda\left(s_{j}\right)\in\mbox{span}\lambda\left(H\right)$,
then 
\begin{eqnarray*}
{\textstyle \left\langle u\cdot S,v\right\rangle =\sum_{j}\alpha_{j}v\left(s_{j}\right)u\left(s_{j}\right)=0} &  & \left(v\in A\left(G\right)\right),
\end{eqnarray*}
so that $u\cdot S=0$. Thus $0=u\cdot T_{\alpha}\overset{w^{*}}{\rightarrow}u\cdot\Psi\left(\lambda\left(s\right)\right)=\Psi\left(u\cdot\lambda\left(s\right)\right)=\Psi\left(\lambda\left(s\right)\right)$
and $\Psi\left(\lambda\left(s\right)\right)=0$.
\end{example}
Let $G$ be a locally compact group. A bounded net $\left(e_{\alpha}\right)_{\alpha}$
in $A\left(G\right)$ or $A_{cb}\left(G\right)$ is called a \emph{$\Delta$-weak
bounded approximate identity} if it converges pointwise to $1_{G}$.
This notion was introduced in \cite{JL} and shown in \cite{L=0000DC}
to be closely related to the existence of natural projections. Reasoning
as in Example \ref{exa:BAI_implies_approximable} shows that a bounded
approximate identity for $A_{cb}\left(G\right)$ is a $\Delta$-weak
one, so that $A_{cb}\left(G\right)$ has $\Delta$-weak bounded approximate
identity whenever $G$ is weakly amenable. The function $1_{G}$ is
$A_{cb}\left(G\right)$-approximable when $A_{cb}\left(G\right)$
has a $\Delta$-weak bounded approximate identity.

The construction of the map $\tau$ above and the proof of Proposition
\ref{pro:chi_H_approximable} may be carried out with $M_{cb}A\left(G\right)$
replaced by $B\left(G\right)$, but, to conclude that $\chi_{H}$
is $B\left(G\right)$-approximable using this result, we require $1_{G}$
to be in the weak$^{*}$ closure of $A\left(G\right)$ in $B\left(G_{d}\right)$.
The proof of Theorem \ref{thm:H_amen_iff_chiH_approx} shows that
the canonical map $A\left(G\right)^{**}\rightarrow A^{d}\left(G\right)$
extending inclusion $A\left(G\right)\subset A^{d}\left(G\right)$
is surjective, so that $1_{G}$ is then the weak$^{*}$ limit of a
bounded net, which is then a $\Delta$-weak bounded approximate identity
for $A\left(G\right)$, implying that $G$ is already amenable \cite[Theorem 5.1]{K=0000DC}.
It is the availability of $\Delta$-weak bounded approximate identities
in $A_{cb}\left(G\right)$ for a larger class of groups \textemdash{}
containing at least the weakly amenable groups \textemdash{} that
is responsible for the utility of Proposition \ref{pro:chi_H_approximable}.
Whether the existence of a $\Delta$-weak bounded approximate identity
for $A_{cb}\left(G\right)$ implies weak amenability of $G$ appears
to be an open question.

For a locally compact group $G$ and closed subgroup $H$, let 
\begin{eqnarray*}
r_{H}:M_{cb}A\left(G_{d}\right)\rightarrow M_{cb}A\left(H_{d}\right) & \mbox{and} & e_{H}:M_{cb}A\left(H_{d}\right)\rightarrow M_{cb}A\left(G_{d}\right)
\end{eqnarray*}
denote respectively the restriction map and the extension by zero.
For a function $f$ on $H$ let $\overset{\circ}{f}$ denote its extension
by zero to $G$. The restriction $r_{H}$ is a complete quotient and
extension $e_{H}$ a complete isometry (\cite[Corollary 6.3]{Spronk}
or \cite[Proposition 4.1]{Stan}), and it is clear that $r_{H}e_{H}=\mbox{id}_{M_{cb}A\left(H_{d}\right)}$
and $e_{H}r_{H}=M_{\chi_{H}}$, the multiplication by $\chi_{H}$.
\begin{lem}
\label{lem:Restr_cont}Let $G$ be a locally compact group and $H$
a closed subgroup. The maps $r_{H}$ and $e_{H}$ are weak$^{*}$
continuous.
\end{lem}
\begin{proof}
If $f=\sum_{j=1}^{n}\alpha_{j}\delta_{x_{j}}\in\mathcal{C}_{c}\left(H_{d}\right)$,
then $\overset{\circ}{f}\in\mathcal{C}_{c}\left(G_{d}\right)$ and
\begin{eqnarray*}
\left\langle r_{H}^{*}\left(f\right),m\right\rangle =\sum_{j=1}^{n}\alpha_{j}m\left(x_{j}\right)=\left\langle \overset{\circ}{f},m\right\rangle  &  & \left(m\in M_{cb}A\left(G_{d}\right)\right),
\end{eqnarray*}
showing that $r_{H}^{*}\left(\mathcal{C}_{c}\left(H_{d}\right)\right)\subset Q\left(G_{d}\right)$.
Since $\mathcal{C}_{c}\left(H_{d}\right)$ is dense in $Q\left(H_{d}\right)$,
it follows that $r_{H}$ is $\mbox{weak}^{*}$ continuous.

Now if $f=\sum_{j=1}^{n}\alpha_{j}\delta_{x_{j}}\in\mathcal{C}_{c}\left(G_{d}\right)$,
then, for $m\in M_{cb}A\left(H_{d}\right)$,
\[
\left\langle e_{H}^{*}\left(f\right),m\right\rangle =\left\langle \overset{\circ}{m},f\right\rangle =\sum_{j=1}^{n}\alpha_{j}\chi_{H}\left(x_{j}\right)m\left(x_{j}\right)=\left\langle \sum_{j=1}^{n}\alpha_{j}\chi_{H}\left(x_{j}\right)\delta_{x_{j}},m\right\rangle ,
\]
and so $\sum_{j=1}^{n}\alpha_{j}\chi_{H}\left(x_{j}\right)\delta_{x_{j}}\in\mathcal{C}_{c}\left(H_{d}\right)$.
Therefore $e_{H}^{*}\left(\mathcal{C}_{c}\left(G_{d}\right)\right)\subset Q\left(H_{d}\right)$
and the claim follows by density, as above.
\end{proof}
\begin{lem}
\label{lem:chiH_mult_AofG_char}Let $G$ be a locally compact group
and $H$ a closed subgroup. The restriction $r_{H}$ maps $M_{cb}^{d}A\left(G\right)$
into $M_{cb}^{d}A\left(H\right)$. In addition, the following are
equivalent:

\begin{enumerate}
\item \label{enu:chi_H_AG}$\chi_{H}A\left(G\right)\subset A_{cb}^{d}\left(G\right)$.
\item \label{enu:Ext_AdofH_in_AdofG}$e_{H}\left(A_{cb}^{d}\left(H\right)\right)\subset A_{cb}^{d}\left(G\right)$.
\item \label{enu:AdofG_is_direct_sum}$A_{cb}^{d}\left(G\right)=I_{A_{cb}^{d}\left(G\right)}\left(G\setminus H\right)\oplus I_{A_{cb}^{d}\left(G\right)}\left(H\right)$
(algebraic direct sum).
\end{enumerate}
\end{lem}
\begin{proof}
Since the restriction of a cb-multiplier of $G$ to the closed subgroup
$H$ yields a cb-multiplier of $H$ \cite[Proposition 1.12]{dCH},
the first claim follows from $\mbox{weak}^{*}$ continuity of $r_{H}$.

\eqref{enu:chi_H_AG} implies \eqref{enu:Ext_AdofH_in_AdofG}: If
$\chi_{H}A\left(G\right)\subset A_{cb}^{d}\left(G\right)$, then,
because $A\left(H\right)=r_{H}\left(A\left(G\right)\right)$,
\[
e_{H}\left(A\left(H\right)\right)=e_{H}\left(r_{H}\left(A\left(G\right)\right)\right)=\chi_{H}A\left(G\right)\subset A_{cb}^{d}\left(G\right)
\]
 and \eqref{enu:Ext_AdofH_in_AdofG} follows by $\mbox{weak}^{*}$
continuity of $e_{H}$.

\eqref{enu:Ext_AdofH_in_AdofG} implies \eqref{enu:AdofG_is_direct_sum}:
If $e_{H}\left(A_{cb}^{d}\left(H\right)\right)\subset A_{cb}^{d}\left(G\right)$,
then given $m\in A_{cb}^{d}\left(G\right)$, the $\mbox{weak}^{*}$
continuity of $r_{H}$ implies $r_{H}\left(m\right)\in A_{cb}^{d}\left(H\right)$
and it follows that $\chi_{H}m=e_{H}r_{H}\left(m\right)\in A_{cb}^{d}\left(G\right)$,
whence $\chi_{G\setminus H}m=m-\chi_{H}m\in A_{cb}^{d}\left(G\right)$
and therefore $m=\chi_{H}m+\chi_{G\setminus H}m\in I_{A_{cb}^{d}\left(G\right)}\left(G\setminus H\right)+I_{A_{cb}^{d}\left(G\right)}\left(H\right)$.
These ideals clearly have trivial intersection.

\eqref{enu:AdofG_is_direct_sum}implies \eqref{enu:chi_H_AG}: Write
$m\in A\left(G\right)$ as $m=m_{1}+m_{2}$ for $m_{1}\in I_{A_{cb}^{d}\left(G\right)}\left(G\setminus H\right)$
and $m_{2}\in I_{A_{cb}^{d}\left(G\right)}\left(H\right)$, in which
case $\chi_{H}m=m_{1}\in A_{cb}^{d}\left(G\right)$.
\end{proof}
\sloppy

When the equivalent conditions of Lemma \ref{lem:chiH_mult_AofG_char}
hold, condition \eqref{enu:Ext_AdofH_in_AdofG} implies that $e_{H}\left(A_{cb}^{d}\left(H\right)\right)=I_{A_{cb}^{d}\left(G\right)}\left(G\setminus H\right)$
and, because $e_{H}$ is isometric, condition \eqref{enu:AdofG_is_direct_sum}
asserts that $A_{cb}^{d}\left(G\right)=A_{cb}^{d}\left(H\right)\oplus I_{A_{cb}^{d}\left(G\right)}\left(H\right)$.
\begin{lem}
\label{lem:Weaker_cond_for_chiH_approximable}Let $G$ be a locally
compact group and $H$ a closed subgroup. If $\chi_{H}A\left(G\right)\subset A_{cb}^{d}\left(G\right)$
and $1_{H}$ is $A_{cb}\left(H\right)$-approximable, then $\chi_{H}$
is $A_{cb}\left(G\right)$-approximable.
\end{lem}
\begin{proof}
It follows from Lemma \ref{lem:chiH_mult_AofG_char}\eqref{enu:Ext_AdofH_in_AdofG}
that $\chi_{H}=e_{H}\left(1_{H}\right)\in A_{cb}^{d}\left(G\right)$.
\end{proof}
Combining the results of this section, we obtain the following.
\begin{thm}
The characteristic function of a closed subgroup $H$ of a locally
compact group $G$ is $A_{cb}\left(G\right)$-approximable when either
of the following conditions is satisfied:
\begin{enumerate}
\item $G_{d}$ has the approximation property.
\item There is a bounded map $\Psi:VN\left(G\right)\rightarrow VN\left(G\right)$
such that $\Psi\left(\lambda\left(s\right)\right)=\chi_{H}\left(s\right)\lambda\left(s\right)$
for $s\in G$, which is satisfied if $\Psi$ is a natural or invariant
projection onto $VN_{H}\left(G\right)$, and $1_{H}$ is $A_{cb}\left(H\right)$-approximable,
which occurs when $H$ is weakly amenable or $H_{d}$ has the approximation
property.
\end{enumerate}
\end{thm}
\begin{proof}
(1) When $G_{d}$ has the approximation property, Proposition \ref{pro:Approx_prop_implies_approximable}
asserts that $A_{cb}^{d}\left(G\right)=M_{cb}A\left(G_{d}\right)$.
Since $M_{cb}A\left(G_{d}\right)$ contains the characteristic functions
of all subgroups of $G$, we have $\chi_{H}\in A_{cb}^{d}\left(G\right)$.

(2) If $\Psi:VN\left(G\right)\rightarrow VN_{H}\left(G\right)$ is
a bounded map such that $\Psi\left(\lambda\left(s\right)\right)=\chi_{H}\left(s\right)\lambda\left(s\right)$
for $s\in G$ and $1_{H}$ is $A_{cb}\left(H\right)$-approximable,
then $\chi_{H}A\left(G\right)\subset A_{cb}^{d}\left(G\right)$ by
Proposition \ref{pro:chi_H_approximable}. Lemma \ref{lem:Weaker_cond_for_chiH_approximable}
then asserts that $\chi_{H}$ is $A_{cb}\left(G\right)$-approximable.
Example \ref{exa:Inv_proj_are_natural} shows that invariant projections
$VN\left(G\right)\rightarrow VN_{H}\left(G\right)$ are natural, and
the latter are by definition bounded map satisfying the condition
of (2). It was noted above that weak amenability of $H$ implies $1_{H}\in A_{cb}^{d}\left(H\right)$,
while Proposition \ref{pro:Approx_prop_implies_approximable} implies
$1_{H}\in A_{cb}^{d}\left(H\right)$ when $H_{d}$ has the approximation
property.
\end{proof}
Theorem 3.7 of \cite{ARS} claims that a bounded approximate indicator
in $B\left(G\right)$ for a closed subgroup $H$ of the locally compact
group $G$ exists whenever $\chi_{H}$ is $B\left(G\right)$-approximable.
The argument establishing this result was found to contain an error
\cite{ARST} and it is not known whether the claim holds. The following
examples show that $\chi_{H}$ may be $M_{cb}A\left(G\right)$-approximable
even when no invariant projection onto $VN_{H}\left(G\right)$ exists,
in which case no bounded approximate indicator for $H$ exists, either,
by Proposition \ref{pro:Approx_ind_yield_inv_proj_AG}.
\begin{example}
The locally compact group $G=SL\left(2,\mathbb{R}\right)$ contains
$H=\mathbb{F}_{2}$ as a closed subgroup. It has recently been shown
that $G_{d}$ is weakly amenable \cite{KL}, so that $\chi_{H}$ is
$A_{cb}\left(G\right)$-approximable by Proposition \ref{pro:Approx_prop_implies_approximable}.
Since $G$ is connected, its group von Neumann algebra is injective
\cite[Corollary 6.9(c)]{Connes}, so there exists a completely bounded
projection $B\left(L^{2}\left(G\right)\right)\rightarrow VN\left(G\right)$.
If a completely bounded projection $VN\left(G\right)\rightarrow VN_{H}\left(G\right)$
existed, then composition would yield a completely bounded projection
$B\left(L^{2}\left(G\right)\right)\rightarrow VN_{H}\left(G\right)$,
implying that $VN_{H}\left(G\right)=VN\left(H\right)$ is an injective
von Neumann algebra \cite{CS}. It would follow that the discrete
group $H$ is amenable \cite[(2.35)]{Paterson}, which is false.

\begin{example}
Let $G=SL\left(2,\mathbb{R}\right)$ and consider the diagonal subgroup
$G_{\Delta}$ of $G\times G$. The Fourier algebra $A\left(G\right)$
is not operator biflat by Corollary 3.7 of \cite{CT}, meaning that
no invariant projection $VN\left(G\times G\right)\rightarrow VN_{G_{\Delta}}\left(G\times G\right)$
exists. But the weak amenability of $G_{d}$, noted in the preceding
example, implies the weak amenability of $\left(G\times G\right)_{d}$,
so that $\chi_{G_{\Delta}}$ is $A_{cb}\left(G\times G\right)$-approximable,
again by Proposition \ref{pro:Approx_prop_implies_approximable}.
\end{example}
\end{example}

\section{\label{sec:The_discr_sep_prop}The discretized $H$-separation property}

In this section, we characterize the approximability of the characteristic
function of a closed subgroup $H$ of a locally compact group $G$
in the spirit of the $H$-separation property of Kaniuth and Lau.
For a closed subgroup $H$ of $G$, let $P_{H}\left(G\right)$ denote
the norm closed convex set $\left\{ u\in P\left(G\right):u\left(H\right)=1\right\} $.
\begin{defn}
(\cite{KanLau}) A locally compact group $G$ is said to have the
\emph{$H$-separation property} for a closed subgroup $H$ if, for
each $s\in G\setminus H$, there exists $u\in P_{H}\left(G\right)$
such that $u\left(s\right)\neq1$.
\end{defn}
It is routine to verify that $G$ has the $H$-separation property
for any open, compact, or normal subgroup $H$, and it was shown by
Forrest \cite{Forrest3} that if $G$ is a SIN group, then $G$ has
the $H$-separation property for every closed subgroup $H$. In \cite{KanLau},
a fixed point argument is used to show that an invariant projection
$VN\left(G\right)\rightarrow VN_{H}\left(G\right)$ exists when the
locally compact group $G$ has the $H$-separation property (it is
noted in Proposition \ref{pro:Approx_ind_yield_inv_proj_AG} below
that the projections arising this way are completely positive, in
particular completely bounded). In fact, the following stronger result
holds.
\begin{prop}
Let $G$ be a locally compact group and $H$ a closed subgroup. Then
$G$ has the $H$-separation property if and only if there exists
a bounded approximate indicator for $H$ in $P_{H}\left(G\right)$.
\end{prop}
\begin{proof}
Suppose that $G$ has the $H$-separation property. The proof of \cite[Proposition 3.1]{KanLau}
constructs an invariant projection $P:VN\left(G\right)\rightarrow VN_{H}\left(G\right)$
that is the weak$^{*}$ operator topology limit of a net $\left(M_{u_{\alpha}}\right)_{\alpha}$,
where $u_{\alpha}\in P_{H}\left(G\right)$ and $M_{u_{\alpha}}:VN\left(G\right)\rightarrow VN\left(G\right)$
is the adjoint of the multiplication map $u\mapsto u_{\alpha}u$ on
$A\left(G\right)$. Let $r_{H}:A\left(G\right)\rightarrow A\left(H\right)$
be the restriction map, which is a surjection satisfying $r_{H}^{*}\left(VN\left(H\right)\right)\subset VN_{H}\left(G\right)$.
Given $u\in A\left(H\right)$, let $\tilde{u}\in A\left(G\right)$
with $r_{H}\left(\tilde{u}\right)=u$, so that for $T\in VN\left(H\right)$,
\begin{eqnarray*}
\left\langle T,ur_{H}\left(u_{\alpha}\right)\right\rangle  & = & \left\langle T,r_{H}\left(\tilde{u}u_{\alpha}\right)\right\rangle \\
 & = & \left\langle M_{u_{\alpha}}\left(r_{H}^{*}\left(T\right)\right),\tilde{u}\right\rangle \\
 & \rightarrow & \left\langle P\left(r_{H}^{*}\left(T\right)\right),\tilde{u}\right\rangle \\
 & = & \left\langle r_{H}^{*}\left(T\right),\tilde{u}\right\rangle \\
 & = & \left\langle T,u\right\rangle .
\end{eqnarray*}
If $w\in I_{A\left(G\right)}\left(H\right)$, then
\begin{eqnarray*}
\left\langle T,wu_{\alpha}\right\rangle =\left\langle M_{u_{\alpha}}\left(T\right),w\right\rangle \rightarrow\left\langle P\left(T\right),w\right\rangle =0 &  & \left(T\in VN\left(G\right)\right)
\end{eqnarray*}
since $P\left(T\right)\in VN_{H}\left(G\right)=I_{A\left(G\right)}\left(H\right)^{\perp}$.
Therefore $ur_{H}\left(u_{\alpha}\right)\rightarrow u$ weakly in
$A\left(H\right)$ for all $u\in A\left(H\right)$ and $wu_{\alpha}\rightarrow0$
weakly in $A\left(G\right)$ for all $w\in I_{A\left(G\right)}\left(H\right)$.
Passing to convex combinations yields a bounded approximate indicator
for $H$ which remains in the convex set $P_{H}\left(G\right)$.

Conversely, if $\left(u_{\alpha}\right)_{\alpha}$ is a bounded approximate
indicator for $H$ in $P_{H}\left(G\right)$, then, given $s\in G\setminus H$,
choose $w\in I_{A\left(G\right)}\left(H\right)$ with $w\left(s\right)=1$,
in which case
\[
\left|u_{\alpha}\left(s\right)\right|=\left|u_{\alpha}\left(s\right)w\left(s\right)\right|\leq\left\Vert u_{\alpha}w\right\Vert _{L^{\infty}\left(G\right)}\leq\left\Vert u_{\alpha}w\right\Vert _{A\left(G\right)}\rightarrow0
\]
implies $u_{\alpha}\left(s\right)\neq1$ eventually.
\end{proof}
For a closed subgroup $H$ of a locally compact group $G$, we now
show that a weaker form of the $H$-separation property, replacing
the algebra $B\left(G\right)$ with $B^{d}\left(G\right)$, characterizes
when $\chi_{H}$ is $B\left(G\right)$-approximable.
\begin{defn}
Let $G$ be a locally compact group and $H$ a closed subgroup. The
group $G$ is said to have the \emph{discretized $H$-separation property}
if, for any $s\in G\setminus H$, there exists $u\in B^{d}\left(G\right)\cap P_{H}\left(G_{d}\right)$
such that $u\left(s\right)\neq1$.
\end{defn}
\begin{prop}
\label{pro:BofG_discretized_sep_prop}Let $G$ be a locally compact
group and $H$ a closed subgroup. Then $G$ has the discretized $H$-separation
property if and only if $\chi_{H}$ is $B\left(G\right)$-approximable.
\end{prop}
\begin{proof}
Suppose that $G$ has the discretized $H$-separation property and
for each $s\in G\setminus H$ let $u_{s}\in B^{d}\left(G\right)\cap P_{H}\left(G_{d}\right)$
with $u_{s}\left(s\right)\neq1$. Replacing $u_{s}$ by $\frac{1}{2}\left(1_{G}+u_{s}\right)$,
which remains in $B^{d}\left(G\right)\cap P_{H}\left(G_{d}\right)$,
we may assume that $\left|u_{s}\left(s\right)\right|<1$. Then the
sequence $\left(u_{s}^{n}\right)_{n\geq1}$ is in $B^{d}\left(G\right)\cap P_{H}\left(G_{d}\right)$
with $\left\Vert u_{s}^{n}\right\Vert _{B\left(G_{d}\right)}=u_{s}^{n}\left(e\right)=1$
and thus has a weak$^{*}$ cluster point $u_{s}^{0}$ in the unit
ball of $B^{d}\left(G\right)$. Then $\left.u_{s}^{0}\right|_{H}=1$
and $\left|u_{s}^{0}\left(s\right)\right|\leq\limsup_{n}\left|u_{s}\left(s\right)\right|^{n}=0$,
so $u_{s}^{0}\left(s\right)=0$. Let $\mathcal{F}$ be the collection
of finite subsets of $G$ and for each $F\in\mathcal{F}$ let $u_{F}=\prod_{s\in F}u_{s}^{0}$.
Ordering $\mathcal{F}$ by inclusion, we have $\left.u_{F}\right|_{H}=1$
and $u_{F}\left(s\right)=0$ eventually for each $s\in G\setminus H$,
so that $u_{F}\overset{ptw}{\rightarrow}\chi_{H}$ and by boundedness
$u_{F}\overset{w^{*}}{\rightarrow}\chi_{H}$ in $B\left(G_{d}\right)$,
whence $\chi_{H}\in B^{d}\left(G\right)$. The converse is clear,
given that the characteristic function of a subgroup is always in
$P_{H}\left(G_{d}\right)$.
\end{proof}
When the locally compact group $G$ is second countable, the $H$-separation
property may also be characterized in terms of a single function on
$G$.
\begin{thm}
Let $G$ be a second countable locally compact group. For a closed
subgroup $H$, the following are equivalent:
\begin{enumerate}
\item $G$ has the $H$-separation property.
\item There is $u\in B\left(G\right)$ of norm one with $\left\{ s\in G:u\left(s\right)=1\right\} =H$.
\item There is $u\in P\left(G\right)$ with $\left\{ s\in G:u\left(s\right)=1\right\} =H$.
\end{enumerate}
\end{thm}
\begin{proof}
(1) implies (2): For $s\in G\setminus H$, let $u_{s}\in P_{H}\left(G\right)$
with $u_{s}\left(s\right)\neq1$ and choose an open neighborhood $U_{s}$
of $s$ with $1\notin u_{s}\left(U_{s}\right)$. Then $\left(U_{s}\right)_{s\in G\setminus H}$
is an open cover of $G\setminus H$, so has a countable subcover $\left(U_{s_{n}}\right)_{n\geq1}$
by $\sigma$-compactness of the open set $G\setminus H$ in $G$.
The function $u=\sum_{n\geq1}2^{-n}u_{s_{n}}$ is in the norm closed
convex set $P_{H}\left(G\right)$, so that $\left\Vert u\right\Vert _{B\left(G\right)}=u\left(e\right)=1$.
Given $s\in G\setminus H$, choose $n$ such that $s\in U_{s_{n}}$,
so $u_{s_{n}}\left(s\right)\neq1$. Since $\left\Vert u_{s_{n}}\right\Vert _{L^{\infty}\left(G\right)}=1$,
we have $\text{Re}u_{s_{n}}\left(s\right)<1$, implying that $\text{Re}u\left(s\right)=\sum_{n\geq1}2^{-n}\text{Re}u_{s_{n}}\left(s\right)<1$
and hence that $u\left(s\right)\neq1$.

(2) implies (3): Replacing $u$ by $\frac{1}{2}\left(u+u^{*}\right)$,
where $u^{*}\left(s\right)=\overline{u\left(s^{-1}\right)}$, we obtain
function of $B\left(G\right)$-norm one (because the $B\left(G\right)$
norm dominates the $L^{\infty}\left(G\right)$ norm) for which we
may write $u=u^{+}-u^{-}$ with $u^{\pm}\in P\left(G\right)$ satisfying
$\left\Vert u\right\Vert _{B\left(G\right)}=\left\Vert u^{+}\right\Vert _{B\left(G\right)}+\left\Vert u^{-}\right\Vert _{B\left(G\right)}$.
Then
\[
u^{+}\left(e\right)-u^{-}\left(e\right)=u\left(e\right)=1=\left\Vert u\right\Vert _{B\left(G\right)}=\left\Vert u^{+}\right\Vert _{B\left(G\right)}+\left\Vert u^{-}\right\Vert _{B\left(G\right)}=u^{+}\left(e\right)+u^{-}\left(e\right)
\]
and consequently $\left\Vert u^{-}\right\Vert _{B\left(G\right)}=u^{-}\left(e\right)=0$,
so that $u=u^{+}\in P_{H}\left(G\right)$.

(3) implies (1): This is clear.
\end{proof}
The amenability of $H$ is known to imply the existence of an invariant
projection $VN\left(G\right)\rightarrow VN_{H}\left(G\right)$ for
any locally compact group $G$ containing $H$ as a closed subgroup
\cite[Corollary 3.7]{Crann} (see also \cite{Derighetti}). We now
show that the former condition may be characterized in terms of a
separation property relative to any such $G$.
\begin{thm}
\label{thm:H_amen_iff_chiH_approx}A locally compact group $H$ is
amenable if and only if $\chi_{H}$ is $A\left(G\right)$-approximable
for some (equivalently, any) locally compact group $G$ containing
$H$ as a closed subgroup.
\end{thm}
\begin{proof}
Fix a locally compact group $G$ that contains $H$ as a closed subgroup.

Suppose that $H$ is amenable. Let $\left(e_{\alpha}\right)_{\alpha}$
be a bounded approximate identity for $A\left(H\right)$ and let $\Psi:VN\left(G\right)\rightarrow VN_{H}\left(G\right)$
be an an invariant projection. Example \ref{exa:Inv_proj_are_natural}
shows that $\Psi\left(\lambda_{G}\left(s\right)\right)=\chi_{H}\left(s\right)\lambda_{G}\left(s\right)$
for all $s\in G$ and the argument of Example \ref{exa:BAI_implies_approximable}
shows that $e_{\alpha}\overset{ptw}{\rightarrow}1_{H}$. Recall that
the adjoint of the restriction $r_{H}:A\left(G\right)\rightarrow A\left(H\right)$
is a $\ast$-isomorphism $\tau=r_{H}^{*}:VN\left(H\right)\rightarrow VN_{H}\left(G\right)$
taking $\lambda_{H}\left(s\right)$ to $\lambda_{G}\left(s\right)$
for all $s\in H$. The composition 
\[
VN\left(H\right)^{*}\overset{\left(\tau^{*}\right)^{-1}}{\rightarrow}VN_{H}\left(G\right)^{*}\overset{\Psi^{*}}{\rightarrow}VN\left(G\right)^{*}
\]
 satisfies, for $s\in G$,
\begin{eqnarray*}
\left\langle \Psi^{*}\left(\tau^{*}\right)^{-1}\left(e_{\alpha}\right),\lambda_{G}\left(s\right)\right\rangle  & = & \left\langle \left(\tau^{*}\right)^{-1}\left(e_{\alpha}\right),\chi_{H}\left(s\right)\lambda_{G}\left(s\right)\right\rangle \\
 & = & \left\langle e_{\alpha},\tau^{-1}\left(\chi_{H}\left(s\right)\lambda_{G}\left(s\right)\right)\right\rangle \\
 & = & \begin{cases}
\left\langle e_{\alpha},\lambda_{H}\left(s\right)\right\rangle , & s\in H\\
0, & s\notin H
\end{cases}\\
 & \rightarrow & \chi_{H}\left(s\right).
\end{eqnarray*}
Let $E$ be a weak$^{*}$ cluster point of the bounded net $\left(\Psi^{*}\left(\tau^{*}\right)^{-1}\left(e_{\alpha}\right)\right)_{\alpha}$
in $VN\left(G\right)^{*}$, so that $\left\langle E,\lambda_{G}\left(s\right)\right\rangle =\chi_{H}\left(s\right)$
for all $s\in G$ by the above computation. Letting $\left(u_{\alpha}\right)_{\alpha}$
be a bounded net in $A\left(G\right)$ converging weak$^{*}$ to $E$,
we have $u_{\alpha}\overset{ptw}{\rightarrow}\chi_{H}$ and therefore
$u_{\alpha}\overset{w^{*}}{\rightarrow}\chi_{H}$ in $B\left(G_{d}\right)$
by boundedness.

Conversely, suppose that $\left(u_{\alpha}\right)_{\alpha}$ is a
net in $A\left(G\right)$ such that $u_{\alpha}\overset{w^{*}}{\rightarrow}\chi_{H}$
in $B\left(G_{d}\right)$. Analogous to the proof of Lemma \ref{lem:Restr_cont},
the restriction $r_{H}:B\left(G_{d}\right)\rightarrow B\left(H_{d}\right)$
is weak$^{*}$ continuous, so that $r_{H}\left(u_{\alpha}\right)\overset{w^{*}}{\rightarrow}1_{H}$
and $1_{H}$ is in the weak$^{*}$ closure of $A\left(H\right)$ in
$B\left(H_{d}\right)$. Viewing $\lambda_{H}$ as a representation
of $H_{d}$, the universal property of $C^{*}\left(H_{d}\right)$
yields a quotient $\ast$-homomorphism $C^{*}\left(H_{d}\right)\rightarrow C_{\delta}^{*}\left(H\right):\omega_{H_{d}}\left(s\right)\mapsto\lambda_{H}\left(s\right)$,
where $\omega_{H_{d}}$ denotes the universal representation of $H_{d}$
and $C_{\delta}^{*}\left(H\right)$ the $C^{*}$-algebra generated
by $\lambda_{H}\left(H\right)$ in $B\left(L^{2}\left(H\right)\right)$,
a subalgebra of $VN\left(H\right)$. Composing with the inclusion,
we obtain a $\ast$-homomorphism $\Psi:C^{*}\left(H_{d}\right)\rightarrow VN\left(H\right)$.
The adjoint $\Psi^{*}:A\left(H\right)^{**}\rightarrow B\left(H_{d}\right)$
is the weak$^{*}$ continuous extension of the inclusion $A\left(H\right)\subset B\left(H_{d}\right)$,
so that $\ker\Psi=A\left(H\right)_{\perp}$ (preannihilator taken
with respect to the $B\left(H_{d}\right)$\textendash $C^{*}\left(H_{d}\right)$
duality). Then $A\left(H\right)_{\perp}=A^{d}\left(H\right)_{\perp}$
is an ideal in $C^{*}\left(H_{d}\right)$ and $\Psi$ drops to an
injective $\ast$-homomorphism $C^{*}\left(H_{d}\right)/A^{d}\left(H\right)_{\perp}\rightarrow VN\left(H\right)$.
This injective $\ast$-homomorphism is isometric, so that its adjoint
$A\left(H\right)^{**}\rightarrow A^{d}\left(H\right)$ is a quotient
map and given $\epsilon>0$, the function $1_{H}$ is the weak$^{*}$
limit in $B\left(H_{d}\right)$ of a net $\left(v_{\alpha}\right)_{\alpha}$
in $A\left(H\right)$ of bound $1+\epsilon$. The adjoint on $B\left(H_{d}\right)$
is a weak$^{*}$ continuous isometry preserving $A\left(H\right)$,
so that we may replace $v_{\alpha}$ by $\frac{1}{2}\left(v_{\alpha}+v_{\alpha}^{*}\right)$
and assume that $v_{\alpha}=v_{\alpha}^{+}-v_{\alpha}^{-}$ for $v_{\alpha}^{\pm}\in P\left(H\right)\cap A\left(H\right)$
with $\left\Vert v_{\alpha}\right\Vert _{B\left(H\right)}=\left\Vert v_{\alpha}^{+}\right\Vert _{B\left(H\right)}+\left\Vert v_{\alpha}^{-}\right\Vert _{B\left(H\right)}$.
Passing to subnets, let $v^{\pm}$ be weak$^{*}$ limits of $\left(v_{\alpha}^{\pm}\right)_{\alpha}$
in $B\left(H_{d}\right)$. Then
\[
1=\lim_{\alpha}v_{\alpha}\left(e\right)=\lim_{\alpha}v_{\alpha}^{+}\left(e\right)-\lim_{\alpha}v_{\alpha}^{-}\left(e\right)=v^{+}\left(e\right)-v^{-}\left(e\right)
\]
and so
\[
v^{-}\left(e\right)=v^{+}\left(e\right)-1=\lim_{\alpha}v_{\alpha}^{+}\left(e\right)-1=\lim_{\alpha}\left\Vert v_{\alpha}^{+}\right\Vert _{B\left(H\right)}-1\leq\epsilon.
\]
Let $\omega_{\epsilon}$ be a weak$^{*}$ cluster point of $\left(v_{\alpha}^{+}\right)_{\alpha}$
in $A\left(H\right)^{**}$, so that $\omega_{\epsilon}$ is a state
on $VN\left(H\right)$ satisfying, for $s\in H$,
\begin{eqnarray*}
\left|\left\langle \omega_{\epsilon},\lambda_{H}\left(s\right)\right\rangle -1\right| & = & \left|\lim_{\alpha}v_{\alpha}^{+}\left(s\right)-\lim_{\alpha}v_{\alpha}\left(s\right)\right|\\
 & = & \lim_{\alpha}\left|v_{\alpha}^{-}\left(s\right)\right|\leq\lim_{\alpha}v_{\alpha}^{-}\left(e\right)=v^{-}\left(e\right)\leq\epsilon.
\end{eqnarray*}
Letting $\omega$ be a weak$^{*}$ cluster point of the states $\left(\omega_{\epsilon}\right)_{\epsilon>0}$
on $VN\left(H\right)$, we have $\left\langle \omega,\lambda_{H}\left(s\right)\right\rangle =1$
for all $s\in H$, and any extension of $\omega$ to a state on $B\left(L^{2}\left(H\right)\right)$
still takes the value $1$ on the unitaries $\lambda_{H}\left(s\right)$
for $s\in H$. The amenability of $H$ follows: by a Cauchy-Schwarz
argument, any extension of $\omega$ to a state on $B\left(L^{2}\left(H\right)\right)$
is invariant under the conjugation action of the unitaries $\lambda_{H}\left(s\right)$
for $s\in H$, whence $\lambda_{H}$ is an amenable representation
of $H$ (see \cite{Bekka}).
\end{proof}
\begin{cor}
Let $G$ be a locally compact group and $H$ a closed subgroup. Then
$H$ is amenable if and only if $G$ has the discretized $H$-separation
property witnessed by functions in $A^{d}\left(G\right)\cap P_{H}\left(G_{d}\right)$.
\end{cor}
\begin{proof}
In the argument establishing Proposition \ref{pro:BofG_discretized_sep_prop},
substituting $\frac{1}{2}\left(u_{s}^{2}+u_{s}\right)$ for the function
$\frac{1}{2}\left(1_{G}+u_{s}\right)$ yields a proof that $G$ has
the desired property if and only if $\chi_{H}$ is $A\left(G\right)$-approximable.
\end{proof}

\section{\textup{\label{sec:Invariant_projections}}Invariant projections
and bounded approximate indicators}

In this section we establish some consequences of the existence of
a bounded approximate indicator for a closed subgroup of a locally
compact group. We first provide the well known argument that this
stronger separation property indeed yields invariant projections.
For a commutative completely contractive Banach algebra $\mathcal{A}$,
let $CB_{\mathcal{A}}\left(\mathcal{A}^{*}\right)$ denote the completely
bounded $\mathcal{A}$-bimodule maps on $\mathcal{A}^{*}$. This space
has compact unit ball when given the \emph{weak$^{*}$ operator topology},
which is determined by the seminorms
\begin{eqnarray*}
\Psi\mapsto\left|\left\langle \Psi\left(\varphi\right),a\right\rangle \right| &  & \left(\varphi\in\mathcal{A}^{*},a\in\mathcal{A}\right).
\end{eqnarray*}

\begin{prop}
\label{pro:Approx_ind_yield_inv_proj_AG}Let $G$ be a locally compact
group and $H$ a closed subgroup. If there is a bounded approximate
indicator for $H$, then there is a completely bounded invariant projection
$VN\left(G\right)\rightarrow VN_{H}\left(G\right)$. If there is a
bounded approximate indicator for $H$ consisting of positive definite
functions, then there is a completely positive invariant projection
$VN\left(G\right)\rightarrow VN_{H}\left(G\right)$.
\end{prop}
\begin{proof}
Let $\left(m_{\alpha}\right)_{\alpha}$ a bounded approximate indicator
for $H$, so that the net of multiplication maps $\left(M_{m_{\alpha}}\right)_{\alpha}$
in $CB_{A\left(G\right)}\left(VN\left(G\right)\right)$ is then bounded
and thus has a $\mbox{weak}^{*}$ operator topology cluster point
$\Psi\in CB\left(VN\left(G\right)\right)$. Passing to a subnet if
necessary, we may assume that $\Psi$ is the limit of this net. For
$u,v\in A\left(G\right)$ and $T\in VN\left(G\right)$,
\[
\left\langle \Psi\left(v\cdot T\right),u\right\rangle =\lim_{\alpha}\left\langle T,m_{\alpha}uv\right\rangle =\left\langle \Psi\left(T\right),uv\right\rangle =\left\langle v\cdot\Psi\left(T\right),u\right\rangle ,
\]
showing that $\Psi$ is invariant. Given $T\in VN_{H}\left(G\right)$,
so that $T=r_{H}^{*}\left(S\right)$ for some $S\in VN\left(H\right)$,
we have for $u\in A\left(G\right)$ that
\begin{eqnarray*}
\left\langle \Psi\left(T\right),u\right\rangle  & = & \lim_{\alpha}\left\langle S,r_{H}\left(uu_{\alpha}\right)\right\rangle \\
 & = & \lim_{\alpha}\left\langle S,r_{H}\left(u\right)r_{H}\left(u_{\alpha}\right)\right\rangle \\
 & = & \left\langle S,r_{H}\left(u\right)\right\rangle \\
 & = & \left\langle T,u\right\rangle ,
\end{eqnarray*}
whence $\Psi$ is the identity on $VN_{H}\left(G\right)$. For $T\in VN\left(G\right)$
and $u\in I_{A\left(G\right)}\left(H\right)$ we have $\left\langle \Psi\left(T\right),u\right\rangle =\lim_{\alpha}\left\langle S,uu_{\alpha}\right\rangle =0$,
so that $\Psi$ maps into $I_{A\left(G\right)}\left(H\right)^{\perp}=VN_{H}\left(G\right)$
and is thus a projection onto $VN_{H}\left(G\right)$.

If the functions $m_{\alpha}$ are in $P\left(G\right)$, then the
maps $M_{m_{\alpha}}$ are completely positive \cite[Proposition 4.2]{dCH}
and by \cite[Theorem 7.4]{Paulsen} their weak$^{*}$ operator topology
cluster point $\Psi$ is also completely positive.
\end{proof}
From the preceding we obtain an analogous result for $A_{cb}\left(G\right)$,
at least when $G$ is a weakly amenable locally compact group.
\begin{prop}
\label{pro:Approx_ind_yield_inv_proj_Acb}Let $G$ be a weakly amenable
locally compact group and $H$ a closed subgroup. If there is a bounded
approximate indicator for $H$, then there is a completely bounded
invariant projection $A_{cb}\left(G\right)^{*}\rightarrow I_{A_{cb}\left(G\right)}\left(H\right)^{\perp}$.
\end{prop}
\begin{proof}
Let $\left(m_{\alpha}\right)_{\alpha}$ an approximate indicator for
$H$. Since $A\left(G\right)$ is an ideal in $M_{cb}A\left(G\right)$,
so too is its closure $A_{cb}\left(G\right)$, so that multiplication
by $m_{\alpha}$ is a completely bounded $A_{cb}\left(G\right)$-module
map on $A_{cb}\left(G\right)$. Denote its adjoint by $M_{m_{\alpha}}$.
Passing to a subnet, we may assume that $\left(M_{m_{\alpha}}\right)_{\alpha}$
has a weak$^{*}$ operator topology limit $\Psi\in CB\left(A_{cb}\left(G\right)^{*}\right)$,
and passing to a further subnet we may assume that the net of maps
$\left(M_{m_{\alpha}}\right)_{\alpha}$ in $CB_{A\left(G\right)}\left(VN\left(G\right)\right)$
also has a weak$^{*}$ operator topology limit $\Psi_{A}$, which
is an invariant projection $VN\left(G\right)\rightarrow VN_{H}\left(G\right)$
by the argument of Proposition \ref{pro:Approx_ind_yield_inv_proj_AG}.
For $u,v\in A_{cb}\left(G\right)$ and $T\in A_{cb}\left(G\right)^{*}$,
we have
\[
\left\langle \Psi\left(v\cdot T\right),u\right\rangle =\lim_{\alpha}\left\langle T,m_{\alpha}uv\right\rangle =\left\langle \Psi\left(T\right),uv\right\rangle =\left\langle v\cdot\Psi\left(T\right),u\right\rangle ,
\]
so $\Psi$ is invariant. Let $\iota:A\left(G\right)\rightarrow A_{cb}\left(G\right)$
be the inclusion. If $T\in A_{cb}\left(G\right)^{*}$ and $u\in A\left(G\right)$,
then
\begin{eqnarray*}
\left\langle \Psi_{A}\iota^{*}\left(T\right),u\right\rangle  & = & \lim_{\alpha}\left\langle \iota^{*}\left(T\right),um_{\alpha}\right\rangle \\
 & = & \lim_{\alpha}\left\langle T,\iota\left(u\right)m_{\alpha}\right\rangle \\
 & = & \left\langle \Psi\left(T\right),\iota\left(u\right)\right\rangle \\
 & = & \left\langle \iota^{*}\Psi\left(T\right),u\right\rangle 
\end{eqnarray*}
and $\Psi_{A}\iota^{*}=\iota^{*}\Psi$ by density of $A\left(G\right)$
in $A_{cb}\left(G\right)$. It follows that 
\[
\iota^{*}\Psi^{2}=\Psi_{A}\iota^{*}\Psi=\Psi_{A}^{2}\iota^{*}=\Psi_{A}\iota^{*}=\iota^{*}\Psi,
\]
which, together with injectivity of $\iota^{*}$, implies that $\Psi^{2}=\Psi$.
If $T\in I_{A_{cb}\left(G\right)}\left(H\right)^{\perp}$, then $\iota^{*}\left(T\right)\in I_{A\left(G\right)}\left(H\right)^{\perp}$
and
\begin{eqnarray*}
\left\langle \Psi\left(T\right),\iota\left(u\right)\right\rangle =\left\langle \Psi_{A}\iota^{*}\left(T\right),u\right\rangle =\left\langle \iota^{*}\left(T\right),u\right\rangle =\left\langle T,\iota\left(u\right)\right\rangle  &  & \left(u\in A\left(G\right)\right),
\end{eqnarray*}
whence $\Psi\left(T\right)=T$, again by density of $A\left(G\right)$.
Therefore $I_{A_{cb}\left(G\right)}\left(H\right)^{\perp}$ is contained
in the range of $\Psi$. Finally, for any $T\in A_{cb}\left(G\right)^{*}$,
if $u\in I_{A\left(G\right)}\left(H\right)$, then $\left\langle \Psi\left(T\right),\iota\left(u\right)\right\rangle =\left\langle \Psi_{A}\iota^{*}\left(T\right),u\right\rangle =0$,
so $\Psi\left(T\right)\in I_{A\left(G\right)}\left(H\right)^{\perp}=\left(\overline{I_{A\left(G\right)}\left(H\right)}^{A_{cb}\left(G\right)}\right)^{\perp}$.
That $A_{cb}\left(G\right)$ has bounded approximate identity implies
every set of synthesis for $A\left(G\right)$ is one for $A_{cb}\left(G\right)$
\cite[Proposition 3.1]{FRS} and, because compactly supported functions
in $A_{cb}\left(G\right)$ are in $A\left(G\right)$, that $H$ is
of spectral synthesis for $A_{cb}\left(G\right)$ is exactly the assertion
that $\overline{I_{A\left(G\right)}\left(H\right)}^{A_{cb}\left(G\right)}=I_{A_{cb}\left(G\right)}\left(H\right)$.
Therefore $\Psi$ has range in $I_{A_{cb}\left(G\right)}\left(H\right)^{\perp}$.
\end{proof}
Note that the arguments of the preceding two propositions yield projections
of completely bounded norm at most the bound on an approximate indicator
for the subgroup. Let $\Lambda_{G}$ denotes the Cowling\textendash Haagerup
constant, that is, the infimum of bounds on approximate identities
for $A_{cb}\left(G\right)$.
\begin{cor}
\label{cor:Approx_ind_yield_bai_ideal_Acb}Let $G$ be a weakly amenable
locally compact group and $H$ a closed subgroup for which an approximate
indicator of bound $C$ exists. The following hold:
\end{cor}
\begin{enumerate}
\item $I_{A_{cb}\left(G\right)}\left(H\right)$ has an approximate identity
of bound $\left(1+C\right)\Lambda_{G}$.
\item An approximate indicator for $H$ of bound $1+\left(1+C\right)\Lambda_{G}$
exists that takes the value one on $H$.
\item An approximate indicator for $H$ of bound $C\Lambda_{G}$ exists
in $A_{cb}\left(G\right)$.
\end{enumerate}
\begin{proof}
(1) The argument of Proposition \ref{pro:Approx_ind_yield_inv_proj_Acb}
yields an invariant projection $A_{cb}\left(G\right)^{*}\rightarrow I_{A_{cb}\left(G\right)}\left(H\right)^{\perp}$
of norm at most $C$ and, because the Banach algebra $A_{cb}\left(G\right)$
has an approximate identity of bound $\Lambda_{G}$, it follows from
\cite[Proposition 6.4]{Forrest1} and its proof that the ideal $I_{A_{cb}\left(G\right)}\left(H\right)$
has an approximate identity of bound $\left(1+C\right)\Lambda_{G}$.

(2) If $\left(e_{\alpha}\right)_{\alpha}$ is an approximate identity
for $I_{A_{cb}\left(G\right)}\left(H\right)$ of bound $\left(1+C\right)\Lambda_{G}$,
then $\left(1_{G}-e_{\alpha}\right)_{\alpha}$ is an approximate indicator
for $H$ with the claimed norm bound.

(3) Let $\left(e_{\alpha}\right)_{\alpha\in A}$ be a bounded approximate
identity for $A_{cb}\left(G\right)$, let $\left(m_{\beta}\right)_{\beta\in B}$
a bounded approximate indicator for $H$, and for $\gamma=\left(\beta,\left(\alpha_{\beta^{\prime}}\right)_{\beta^{\prime}\in B}\right)\in B\times A^{B}$
set $u_{\gamma}=m_{\beta}e_{\alpha_{\beta}}$, which is in the ideal
$A_{cb}\left(G\right)$ of $M_{cb}A\left(G\right)$. Giving $B\times A^{B}$
the product order, for $u\in A\left(H\right)$ and $w\in I_{A\left(G\right)}\left(H\right)$
we have the norm limits 
\[
\lim_{\gamma\in B\times A^{B}}r_{H}\left(u_{\gamma}\right)u=\lim_{\beta}\lim_{\alpha}r_{H}\left(m_{\beta}e_{\alpha}\right)u=\lim_{\beta}r_{H}\left(m_{\beta}\right)u=u
\]
and 
\[
\lim_{\gamma\in B\times A^{B}}u_{\gamma}w=\lim_{\beta}\lim_{\alpha}m_{\beta}e_{\alpha}w=\lim_{\beta}m_{\beta}w=0
\]
by \cite[p. 69]{Kelley}, hence $\left(u_{\gamma}\right)_{\gamma\in B\times A^{B}}$
is a bounded approximate indicator for $H$ of norm bound $\sup_{\alpha}\left\Vert e_{\alpha}\right\Vert _{M_{cb}A\left(G\right)}\sup_{\beta}\left\Vert m_{\beta}\right\Vert _{M_{cb}A\left(G\right)}$.
\end{proof}

\section{\label{sec:Averaging_over_subgroups}Convergence of cb-multipliers
and averaging over closed subgroups}

Fix a locally compact group $G$ and a closed subgroup $H$. It is
folklore that the convergence properties of nets of cb-multipliers
can be improved by convolving them with probability measures in $\mathcal{C}_{c}\left(G\right)$.
For example, Knudby recently recorded the following, the second part
of which originates in an argument of Cowling and Haagerup \cite[Proposition 1.1]{CH}.
If a net $\left(m_{\alpha}\right)_{\alpha}$ of functions on a topological
space converges uniformly on compact sets to a function $m$, we write
$m_{\alpha}\overset{ucs}{\rightarrow}m$.
\begin{thm}
\label{thm:Knudby_lemma}\emph{(}\cite[Lemma B.2]{Knudby}\emph{)}
Let $\left(m_{\alpha}\right)_{\alpha}$ be a bounded net in $M_{cb}A\left(G\right)$,
$m\in M_{cb}A\left(G\right)$, and let $f\in\mathcal{C}_{c}\left(G\right)$
be such that $f\geq0$ and $\int_{G}f=1$. Convolution on the left
with $f$ is a contraction on $M_{cb}A\left(G\right)$ and the following
hold:
\begin{enumerate}
\item If $m_{\alpha}\overset{w^{*}}{\rightarrow}m$ in $M_{cb}A\left(G\right)$,
then $f\ast m_{\alpha}\overset{ucs}{\rightarrow}f\ast m$.
\item If $m_{\alpha}\overset{ucs}{\rightarrow}m$, then $\left\Vert \left(f\ast m_{\alpha}\right)u-\left(f\ast m\right)u\right\Vert _{A\left(G\right)}\rightarrow0$
for all $u\in A\left(G\right)$.
\end{enumerate}
\end{thm}
In this section, we develop an analogue of the convolution technique
relative to a closed subgroup. Fix a function $f\in\mathcal{C}_{c}\left(H\right)$
such that $f\geq0$ and $\int_{H}f=1$. For any function $f$ on $G$
and $s,t\in G$, let $_{s}f\left(t\right)=f\left(st\right)$.
\begin{defn}
For $u\in\mathcal{C}_{b}\left(G\right)$, define a function $\Omega_{f}\left(u\right)$
on $G$ by the formula 
\begin{eqnarray*}
\Omega_{f}\left(u\right)\left(s\right)=\int_{H}f\left(h\right)u\left(h^{-1}s\right)dh &  & \left(s\in G\right).
\end{eqnarray*}
We will show that $\Omega_{f}$ defines a bounded map on $M_{cb}A\left(G\right)$.
For a Hilbert space $\mathcal{H}$, let $\mathcal{C}_{c}\left(G,\mathcal{H}\right)$
and $\mathcal{C}_{b}\left(G,\mathcal{H}\right)$ denote the continuous
functions $G\rightarrow\mathcal{H}$ that are compactly supported
and bounded, respectively.
\end{defn}
\begin{lem}
\label{lem:Cts_cpt_supp_are_ucts}Let $\mathcal{H}$ be a Hilbert
space. If $u\in\mathcal{C}_{c}\left(G,\mathcal{H}\right)$ then for
any $\epsilon>0$ there is an open neighborhood $U$ of the identity
$e$ such that $\sup_{t\in G}\left\Vert u\left(st\right)-u\left(t\right)\right\Vert <\epsilon$
for all $s\in U$.
\end{lem}
\begin{proof}
The standard proof in the case that $\mathcal{H}=\mathbb{C}$, for
example \cite[Proposition 2.6]{Folland1}, works for any Hilbert space.
\end{proof}
\begin{lem}
\label{lem:A_unif_cont_lemma}Let $\mathcal{H}$ be a Hilbert space.
If $u\in\mathcal{C}_{b}\left(G,\mathcal{H}\right)$, $s_{0}\in G$,
and $\epsilon>0$, then there is an open neighborhood $U$ of $s_{0}$
in $G$ such that $\sup_{h\in H}\left\Vert f\left(h\right)u\left(sh\right)-f\left(h\right)u\left(s_{0}h\right)\right\Vert <\epsilon$
for all $s\in U$.
\end{lem}
\begin{proof}
If $u=0$, then the claim trivially holds, so assume $u\neq0$. Since
$H$ is closed in $G$, the function $f$ extends to a continuous
compactly supported function $f^{\prime}$ on $G$. Assume that $s_{0}=e$.
Since $f^{\prime}u$ is compactly supported, Lemma \ref{lem:Cts_cpt_supp_are_ucts}
yields an open neighborhood $U$ of $e$ such that 
\begin{eqnarray*}
\sup_{t\in G}\left\Vert f^{\prime}u\left(st\right)-f^{\prime}u\left(t\right)\right\Vert <\frac{\epsilon}{2} & \mbox{and} & \sup_{t\in G}\left|f^{\prime}\left(st\right)-f^{\prime}\left(t\right)\right|<\frac{\epsilon}{2\left\Vert u\right\Vert _{\infty}}
\end{eqnarray*}
 for all $s\in U$. Then
\begin{eqnarray*}
\sup_{h\in H}\left\Vert f\left(h\right)u\left(sh\right)-f\left(h\right)u\left(h\right)\right\Vert  & \leq & \sup_{t\in G}\left\Vert f^{\prime}\left(t\right)u\left(st\right)-f^{\prime}\left(t\right)u\left(t\right)\right\Vert \\
 & \leq & \sup_{t\in G}\left(\left\Vert f^{\prime}\left(t\right)u\left(st\right)-f^{\prime}u\left(st\right)\right\Vert +\right.\\
 &  & \left.\,\,\,\,\,\,\,\,\,\,\,\,\left\Vert f^{\prime}u\left(st\right)-f^{\prime}\left(t\right)u\left(t\right)\right\Vert \right)\\
 & < & \left\Vert u\right\Vert _{\infty}\sup_{t\in G}\left|f^{\prime}\left(st\right)-f^{\prime}\left(t\right)\right|+\frac{\epsilon}{2}<\epsilon,
\end{eqnarray*}
for all $s\in U$. For $s_{0}\neq e$, the above argument with $u$
replaced by $_{s_{0}}u$ yields a neighborhood $U$ of $e$ and $s_{0}U$
is then the desired neighborhood of $s_{0}$.
\end{proof}
\begin{prop}
\label{pro:Averaging_Fourier_multipliers}If $u\in M_{cb}A\left(G\right)$,
then $\Omega_{f}\left(u\right)\in M_{cb}A\left(G\right)$ with $\left\Vert \Omega_{f}\left(u\right)\right\Vert _{M_{cb}A\left(G\right)}\leq\left\Vert u\right\Vert _{M_{cb}A\left(G\right)}$
and $r_{H}\Omega_{f}\left(u\right)=f\ast r_{H}\left(u\right)$.
\end{prop}
\begin{proof}
Let $u\in M_{cb}A\left(G\right)$ and apply Gilbert's theorem to obtain
a Hilbert space $\mathcal{H}$ and functions $P,Q\in\mathcal{C}_{b}\left(G,\mathcal{H}\right)$
such that $u\left(s^{-1}t\right)=\left\langle P\left(t\right)|Q\left(s\right)\right\rangle $
for all $s,t\in G$. Then, for $s,t\in G$,
\[
{\displaystyle \Omega_{f}\left(u\right)\left(s^{-1}t\right)=\int_{H}f\left(h\right)u\left(h^{-1}s^{-1}t\right)dh=\left\langle P\left(t\right)\left|\int_{H}f\left(h\right)Q\left(sh\right)dh\right.\right\rangle }.
\]
We show that $q\left(s\right)=\int_{H}f\left(h\right)Q\left(sh\right)dh$
defines a bounded continuous function on $G$, from which it will
follow that $\Omega_{f}\left(u\right)$ is in $M_{cb}A\left(G\right)$,
again by Gilbert's theorem. Define $Q^{\prime}:G\rightarrow L^{1}\left(H,\mathcal{H}\right)$
by $Q^{\prime}\left(s\right)=f\left(_{s}Q\right)$, which maps into
$L^{1}\left(H,\mathcal{H}\right)$ since $f$ has compact support.
Set $K=\mbox{supp}f$ and let $\left|K\right|$ denote the Haar measure
of $K$, which is nonzero and finite by continuity of the nonzero,
compactly supported function $f$. Given $s_{0}\in G$ and $\epsilon>0$,
Lemma \ref{lem:A_unif_cont_lemma} yields an open neighborhood $U$
of $s_{0}$ in $G$ such that 
\[
\left\Vert Q^{\prime}\left(s\right)-Q^{\prime}\left(s_{0}\right)\right\Vert _{L^{\infty}\left(H,\mathcal{H}\right)}=\sup_{h\in H}\left\Vert f\left(h\right)Q\left(sh\right)-f\left(h\right)Q\left(s_{0}h\right)\right\Vert <\frac{\epsilon}{\left|K\right|}
\]
for all $s\in U$. Since $Q^{\prime}\left(s\right)$ is supported
in $K$ for every $s\in G$, it follows that
\begin{eqnarray*}
\left\Vert Q^{\prime}\left(s\right)-Q^{\prime}\left(s_{0}\right)\right\Vert _{L^{1}\left(H,\mathcal{H}\right)} & = & \left\Vert \chi_{K}\left(Q^{\prime}\left(s\right)-Q^{\prime}\left(s_{0}\right)\right)\right\Vert _{L^{1}\left(H,\mathcal{H}\right)}\\
 & \leq & \left\Vert \chi_{K}\right\Vert _{L^{1}\left(H,\mathcal{H}\right)}\left\Vert Q^{\prime}\left(s\right)-Q^{\prime}\left(s_{0}\right)\right\Vert _{L^{\infty}\left(H,\mathcal{H}\right)}\\
 & < & \epsilon
\end{eqnarray*}
for all $s\in U$. Thus $Q^{\prime}$ is continuous and so too is
$q$, the latter being the composition of $Q^{\prime}$ with the bounded
map $L^{1}\left(H,\mathcal{H}\right)\rightarrow\mathcal{H}:g\mapsto\int_{H}g$.

Using that $f$ is nonnegative with mass one, if $s\in G$, then $\left\Vert q\left(s\right)\right\Vert \leq\int_{H}f\left(h\right)\left\Vert Q\left(sh\right)\right\Vert dh\leq\left\Vert Q\right\Vert _{\infty}$,
so $q$ is bounded with $\left\Vert q\right\Vert _{\infty}\leq\left\Vert Q\right\Vert _{\infty}$.
By the norm characterization of Gilbert's theorem, $\left\Vert \Omega_{f}\left(u\right)\right\Vert _{M_{cb}A\left(G\right)}\leq\left\Vert P\right\Vert _{\infty}\left\Vert q\right\Vert _{\infty}\leq\left\Vert P\right\Vert _{\infty}\left\Vert Q\right\Vert _{\infty}$
and since $P$, $Q$, and $\mathcal{H}$ are an arbitrary representation
of $u$, we conclude that $\left\Vert \Omega_{f}\left(u\right)\right\Vert _{M_{cb}A\left(G\right)}\leq\left\Vert u\right\Vert _{M_{cb}A\left(G\right)}$.
Finally, we have for $s\in H$ that
\begin{eqnarray*}
r_{H}\Omega_{f}\left(u\right)\left(s\right) & = & \int_{H}f\left(h\right)u\left(h^{-1}s\right)dh\\
 & = & \int_{H}f\left(h\right)r_{H}\left(u\right)\left(h^{-1}s\right)dh\\
 & = & f\ast r_{h}\left(u\right)\left(s\right).
\end{eqnarray*}
\end{proof}
\begin{thm}
\label{thm:Averaging_Fourier_multipliers_conv}Let $\left(m_{\alpha}\right)_{\alpha}$
be a bounded net in $M_{cb}A\left(G\right)$ and let $m\in M_{cb}A\left(H\right)$.
The following hold:
\begin{enumerate}
\item If $r_{H}\left(m_{\alpha}\right)\overset{w^{*}}{\rightarrow}m$ in
$M_{cb}A\left(H\right)$, then $r_{H}\Omega_{f}\left(m_{\alpha}\right)\overset{ucs}{\rightarrow}f\ast m$.
\item If $r_{H}\left(m_{\alpha}\right)\overset{ucs}{\rightarrow}m$, then
$\left\Vert r_{H}\Omega_{f}\left(m_{\alpha}\right)u-\left(f\ast m\right)u\right\Vert _{A\left(H\right)}\rightarrow0$
for all $u\in A\left(H\right)$.
\end{enumerate}
\end{thm}
\begin{proof}
These follow immediately from Theorem \ref{thm:Knudby_lemma} and
Proposition \ref{pro:Averaging_Fourier_multipliers}.
\end{proof}
In our applications, the preceding theorem will be applied with $m=1_{H}$,
which is fixed under convolution with $f$ on the left. We list some
additional properties that the map $\Omega_{f}$ enjoys.
\begin{enumerate}
\item An argument very similar to that establishing Proposition \ref{pro:Averaging_Fourier_multipliers}
shows that $\Omega_{f}\left(u\right)$ is bounded and continuous for
any bounded continuous function $u$ on $G$.
\item If $u=\left\langle \pi\left(\cdot\right)\xi|\eta\right\rangle $ is
a coefficient of the unitary representation $\pi$ of $G$, then
\begin{eqnarray*}
\Omega_{f}\left(u\right)\left(s\right) & = & \int_{H}f\left(h\right)\left\langle \pi\left(s\right)\xi|\pi\left(h\right)\eta\right\rangle dh\\
 & = & \left\langle \pi\left(s\right)\xi\left|\left(\int_{H}f\left(h\right)\pi\left(h\right)dh\right)\eta\right.\right\rangle ,
\end{eqnarray*}
so $\Omega_{f}\left(u\right)\in B\left(G\right)$, and from $\left\Vert \int_{H}f\left(h\right)\pi\left(h\right)dh\right\Vert \leq\int_{H}f\left(h\right)\left\Vert \pi\left(h\right)\right\Vert dh=1$
it follows that $\left\Vert \Omega_{f}\left(u\right)\right\Vert _{B\left(G\right)}\leq\left\Vert u\right\Vert _{B\left(G\right)}$.
Thus $\Omega_{f}$ restricts to a contraction on $B\left(G\right)$
and moreover restricts to a contraction on $A\left(G\right)$, since
$\Omega_{f}\left(u\right)$ is a coefficient of the same representation
as $u$.
\item An argument similar to that establishing the weak$^{*}$ continuity
of the map $\Phi_{f}$ in the proof of \cite[Lemma 1.16]{HK} shows
that $\Omega_{f}$ is weak$^{\ast}$ continuous on $M_{cb}A\left(G\right)$
with preadjoint mapping $g\in L^{1}\left(G\right)$ to the $L^{1}\left(G\right)$
function $s\mapsto\int_{H}f\left(h\right)g\left(hs\right)dh$.
\end{enumerate}
Say that a net $\left(m_{\alpha}\right)_{\alpha}$ of functions on
a topological space $X$ \emph{converges locally eventually to zero}
\emph{on $A\subset X$} and write $m_{\alpha}\overset{le}{\rightarrow}0$
if for any compact subset $K$ of $A$ there is an index $\alpha_{0}$
such that $\left.m_{\alpha}\right|_{K}=0$ for all $\alpha\geq\alpha_{0}$.
\begin{prop}
\label{pro:Conditions_for_BAI}If $\left(m_{\alpha}\right)_{\alpha}$
is a bounded net in $M_{cb}A\left(G\right)$ and $m_{\alpha}^{\prime}=\Omega_{f}\left(m_{\alpha}\right)$,
then the net $\left(m_{\alpha}^{\prime}\right)_{\alpha}$ has the
same norm bound as $\left(m_{\alpha}\right)_{\alpha}$ and
\begin{enumerate}
\item if $r_{H}\left(m_{\alpha}\right)\overset{ucs}{\rightarrow}1_{H}$,
then $\left\Vert u\cdot r_{H}\left(m_{\alpha}^{\prime}\right)-u\right\Vert _{A\left(H\right)}\rightarrow0$
for all $u\in A\left(H\right)$, and
\item if $m_{\alpha}\overset{le}{\rightarrow}0$ on $G\setminus H$, then
$m_{\alpha}^{\prime}\overset{le}{\rightarrow}0$ on $G\setminus H$.
\end{enumerate}
If the bounded net $\left(m_{\alpha}\right)_{\alpha}$ satisfies the
hypotheses of both (1) and (2), then $\left(m_{\alpha}^{\prime}\right)_{\alpha}$
is a bounded approximate indicator for $H$.

\end{prop}
\begin{proof}
The claim regarding norm bounds holds since the map $\Omega_{f}$
of Section \ref{sec:Averaging_over_subgroups} is a contraction on
$M_{cb}A\left(G\right)$.

(1) If $r_{H}\left(m_{\alpha}\right)\overset{ucs}{\rightarrow}1_{H}$,
then, since restriction is a contraction from $M_{cb}A\left(G\right)$
into $M_{cb}A\left(H\right)$, the net $\left(r_{H}\left(m_{\alpha}\right)\right)_{\alpha}$
is bounded and (1) follows from Theorem \ref{thm:Averaging_Fourier_multipliers_conv}.

(2) Suppose that $m_{\alpha}\overset{le}{\rightarrow}0$ on $G\setminus H$.
Let $K\subset G\setminus H$ be compact and choose $\alpha_{0}$ such
that $\alpha\geq\alpha_{0}$ implies $m_{\alpha}=0$ on the compact
set $\left(\mbox{supp}\left(f\right)\right)^{-1}K$. For $\alpha\ge\alpha_{0}$,
if $s\in K$ and $h\in H$, then $f\left(h\right)m_{\alpha}\left(h^{-1}s\right)=0$
since either $h\notin\mbox{supp}\left(f\right)$ or $h^{-1}s\in\left(\mbox{supp}\left(f\right)\right)^{-1}K$,
implying that $m_{\alpha}^{\prime}\left(s\right)=\int_{H}f\left(h\right)m_{\alpha}\left(h^{-1}s\right)dh=0$.
Therefore $m_{\alpha}^{\prime}=0$ on $K$, for all $\alpha\geq\alpha_{0}$.

If $\left(m_{\alpha}\right)_{\alpha}$ satisfies the hypotheses of
both (1) and (2), then $\left(m_{\alpha}^{\prime}\right)_{\alpha}$
satisfies the first condition of Definition \ref{def:Subgroup_approx_properties}.
If $u\in I_{A\left(G\right)}\left(H\right)$ has compact support,
then $m_{\alpha}^{\prime}u=0$ eventually by (2), so certainly $\left\Vert um_{\alpha}^{\prime}\right\Vert _{A\left(G\right)}\rightarrow0$.
Using that $H$ is of spectral synthesis in $A\left(G\right)$, if
$u\in I_{A\left(G\right)}\left(H\right)$ is arbitrary, then given
$\epsilon>0$ choose $u_{0}\in I_{A\left(G\right)}\left(H\right)$
of compact support with $\left\Vert u-u_{0}\right\Vert _{A\left(G\right)}<\epsilon$.
For sufficiently large $\alpha$, 
\[
\left\Vert um_{\alpha}^{\prime}\right\Vert _{A\left(G\right)}\leq\left\Vert u_{0}m_{\alpha}^{\prime}\right\Vert _{A\left(G\right)}+\left\Vert u_{0}m_{\alpha}^{\prime}-um_{\alpha}^{\prime}\right\Vert _{A\left(G\right)}<\epsilon\left\Vert m_{\alpha}^{\prime}\right\Vert _{M_{cb}A\left(G\right)},
\]
and thus $\left\Vert um_{\alpha}^{\prime}\right\Vert _{A\left(G\right)}\rightarrow0$
by boundedness of $\left(m_{\alpha}^{\prime}\right)_{\alpha}$.
\end{proof}
Proposition \ref{pro:Conditions_for_BAI} allows one to obtain approximate
indicators consisting of cb-multipliers by verifying the same conditions
that yielded approximate indicators in \cite{ARS}.

\end{document}